\documentclass[10pt,twoside]{article}

\usepackage[dvips,ps,all]{xy}
\xyoption{line}
\xyoption{color}
\xyoption{frame}

\usepackage{a4wide}

\usepackage{amsmath,amsfonts,amssymb,amsthm}

\usepackage{color}
\definecolor{Bordeaux}{rgb}{0.545, 0.137, 0.137}
\definecolor{BleuGris}{rgb}{0.212, 0.392, 0.545}
\definecolor{Chocolat}{rgb}{0.36, 0.2, 0.09}
\definecolor{BleuTresFonce}{rgb}{0.215, 0.215, 0.36}

\usepackage[colorlinks,final,hyperindex]{hyperref}
\hypersetup{citecolor=BleuTresFonce, linkcolor=Chocolat, urlcolor=BleuTresFonce}

\usepackage{pxfonts}

\DeclareSymbolFont{rsfscript}{OMS}{rsfs}{m}{n}
\DeclareSymbolFontAlphabet{\mathrsfs}{rsfscript}

\DeclareFontFamily{OMS}{rsfs}{\skewchar\font'177}
\DeclareFontShape{OMS}{rsfs}{m}{n}{%
      <5> rsfs5
      <6> <7> rsfs7
      <8> <9> <10> rsfs10
      <10.95> <12> <14.4> <17.28> <20.74> <24.88> rsfs10
      }{}

\def\calA{\mathrsfs{A}}

\def\calF{\mathrsfs{F}}
\def\calG{\mathrsfs{G}}
\def\calH{\mathrsfs{H}}

\def\calM{\mathrsfs{M}}
\def\calO{\mathrsfs{O}}
\def\calP{\mathrsfs{P}}
\def\calQ{\mathrsfs{Q}}
\def\calR{\mathrsfs{R}}
\def\calV{\mathrsfs{V}}

\DeclareMathAlphabet{\mathbbold}{U}{bbold}{m}{n}

\def\k{\mathbbold{k}}

\newcommand{\LYY}[4][{}]{\ensuremath{
 \xygraph{
!{<0pt,0pt>;<4pt,0pt>:<0pt,-4pt>::}
!{(1,4)}*{\scriptscriptstyle #1}
!{(1,3)}="a"
!{(1,1)}*{\scriptscriptstyle{\circ}}="b"
!{(-1,-1)}*{\scriptscriptstyle{\circ}}="c"
!{(3,-1)}="d"
!{(3,-2)}*{\scriptscriptstyle #4}
!{(-3,-3)}="e"
!{(-3,-4)}*{\scriptscriptstyle #2}
!{(1,-3)}="f"
!{(1,-4)}*{\scriptscriptstyle #3}
"b"-"c"
"b"-"d"
"c"-"e"
"c"-"f"
}
}}

\newcommand{\LYYb}[4][{}]{\ensuremath{
 \xygraph{
!{<0pt,0pt>;<4pt,0pt>:<0pt,-4pt>::}
!{(1,4)}*{\scriptscriptstyle #1}
!{(1,3)}="a"
!{(1,1)}*{\scriptscriptstyle{\circ}}="b"
!{(-1,-1)}*{\scriptscriptstyle{\circ}}="c"
!{(3,-1)}="d"
!{(3,-2)}*{\scriptscriptstyle #4}
!{(-3,-3)}="e"
!{(-3,-4)}*{\scriptscriptstyle #2}
!{(1,-3)}="f"
!{(1,-4)}*{\scriptscriptstyle #3}
"b"-@[|(3)]"c"
"b"-@[|(3)]"d"
"c"-@[|(3)]"e"
"c"-@[|(3)]"f"
}
}}

\newcommand{\RYY}[4][{}]{\ensuremath{
 \xygraph{
!{<0pt,0pt>;<4pt,0pt>:<0pt,-4pt>::}
!{(5,-4)}*{\scriptscriptstyle #3}
!{(1,3)}="a"
!{(1,1)}*{\scriptscriptstyle{\circ}}="b"
!{(1,-4)}*{\scriptscriptstyle #2}
!{(-1,-1)}="c"
!{(3,-1)}*{\scriptscriptstyle{\circ}}="d"
!{(5,-3)}="e"
!{(-1,-2)}*{\scriptscriptstyle #4}
!{(1,-3)}="f"
!{(-1,-4)}*{\scriptscriptstyle #1}
"b"-"c"
"b"-"d"
"d"-"e"
"d"-"f"
}
}}

\newcommand{\LYRYAB}[5][{}]{\ensuremath{
 \xygraph{
!{<0pt,0pt>;<5pt,0pt>:<0pt,-5pt>::}
!{(1,4)}*{\scriptscriptstyle #1}
!{(1,3)}="a"
!{(1,1)}*{\scriptscriptstyle{\circ}}="b"
!{(-1,-1)}*{\scriptscriptstyle{\bullet}}="c"
!{(3,-1)}*{\scriptscriptstyle{\circ}}="d"
!{(-2,-4)}*{\scriptscriptstyle #2}
!{(-2,-3)}="e"
!{(0,-4)}*{\scriptscriptstyle #3}
!{(0,-3)}="f"
!{(2,-3)}="g"
!{(4,-3)}="h"
!{(2,-4)}*{\scriptscriptstyle #4}
!{(4,-4)}*{\scriptscriptstyle #5}
"b"-"c"
"b"-"d"
"c"-"e"
"c"-"f"
"d"-"g"
"d"-"h"
}
}}

\newcommand{\XYex}[0]{\ensuremath{
\xygraph{!{<0pt,0pt>;<7pt,0pt>:<0pt,-7pt>::}
!{(-1,-1)}*{\scriptscriptstyle{\circ}}="c"
!{(-3,-3)}*{\scriptscriptstyle{\bullet}}="e"
!{(1,-3)}*{\scriptscriptstyle{\circ}}="f"
!{(-4.4,-5)}*{\scriptscriptstyle{\circ}}="g"
!{(-1.9,-5)}="h"
!{(-0.4,-5)}*{\scriptscriptstyle{\bullet}}="i"
!{(2.4,-5)}*{\scriptscriptstyle{\circ}}="j"
!{(-5.4,-7)}="k"
!{(-3.4,-7)}="l"
!{(-1.4,-7)}="m"
!{(0.4,-7)}="n"
!{(1.4,-7)}="o"
!{(3.4,-7)}="p"
!{(-5.4,-8)}*{\scriptscriptstyle 1}
!{(-3.4,-8)}*{\scriptscriptstyle 3}
!{(-1.4,-8)}*{\scriptscriptstyle 2}
!{(0.4,-8)}*{\scriptscriptstyle 7}
!{(1.4,-8)}*{\scriptscriptstyle 4}
!{(3.4,-8)}*{\scriptscriptstyle 6}
!{(-1.9,-6)}*{\scriptscriptstyle 5}
"c"-"e"
"c"-"f"
"e"-"g"
"e"-"h"
"f"-"i"
"f"-"j"
"g"-"k"
"g"-"l"
"i"-"m"
"i"-"n"
"j"-"o"
"j"-"p"
} 
}}

\newcommand{\XYexa}[0]{\ensuremath{
\xygraph{!{<0pt,0pt>;<7pt,0pt>:<0pt,-7pt>::}
!{(-1,-1)}*{\scriptscriptstyle{\circ}}="c"
!{(-3,-3)}*{\scriptscriptstyle{\bullet}}="e"
!{(1,-3)}*{\scriptscriptstyle{\circ}}="f"
!{(-4.4,-5)}*{\scriptscriptstyle{\circ}}="g"
!{(-1.9,-5)}="h"
!{(-0.4,-5)}*{\scriptscriptstyle{\bullet}}="i"
!{(2.4,-5)}*{\scriptscriptstyle{\circ}}="j"
!{(-5.4,-7)}="k"
!{(-3.4,-7)}="l"
!{(-1.4,-7)}="m"
!{(0.4,-7)}="n"
!{(1.4,-7)}="o"
!{(3.4,-7)}="p"
!{(-5.4,-8)}*{\scriptscriptstyle 1}
!{(-3.4,-8)}*{\scriptscriptstyle 3}
!{(-1.4,-8)}*{\scriptscriptstyle 2}
!{(0.4,-8)}*{\scriptscriptstyle 7}
!{(1.4,-8)}*{\scriptscriptstyle 4}
!{(3.4,-8)}*{\scriptscriptstyle 6}
!{(-1.9,-6)}*{\scriptscriptstyle 5}
"c"-@[|(3)]"e"
"c"-@[|(3)]"f"
"e"-@[|(3)]"g"
"e"-@[|(3)]"h"
"f"-@[|(3)]"i"
"f"-@[|(3)]"j"
"g"-"k"
"g"-"l"
"i"-"m"
"i"-"n"
"j"-"o"
"j"-"p"
}
}}

\newcommand{\XYexb}[0]{\ensuremath{
\xygraph{!{<0pt,0pt>;<7pt,0pt>:<0pt,-7pt>::}
!{(-1,-1)}*{\scriptscriptstyle{\circ}}="c"
!{(-3,-3)}*{\scriptscriptstyle{\bullet}}="e"
!{(1,-3)}*{\scriptscriptstyle{\circ}}="f"
!{(-4.4,-5)}*{\scriptscriptstyle{\circ}}="g"
!{(-1.9,-5)}="h"
!{(-0.4,-5)}*{\scriptscriptstyle{\bullet}}="i"
!{(2.4,-5)}*{\scriptscriptstyle{\circ}}="j"
!{(-5.4,-7)}="k"
!{(-3.4,-7)}="l"
!{(-1.4,-7)}="m"
!{(0.4,-7)}="n"
!{(1.4,-7)}="o"
!{(3.4,-7)}="p"
!{(-5.4,-8)}*{\scriptscriptstyle 1}
!{(-3.4,-8)}*{\scriptscriptstyle 3}
!{(-1.4,-8)}*{\scriptscriptstyle 2}
!{(0.4,-8)}*{\scriptscriptstyle 7}
!{(1.4,-8)}*{\scriptscriptstyle 4}
!{(3.4,-8)}*{\scriptscriptstyle 6}
!{(-1.9,-6)}*{\scriptscriptstyle 5}
"f"-@[|(3)]"i"
"f"-@[|(3)]"j"
"i"-@[|(3)]"m"
"i"-@[|(3)]"n"
"j"-@[|(3)]"o"
"j"-@[|(3)]"p"
"c"-"e"
"c"-"f"
"e"-"g"
"e"-"h"
"g"-"k"
"g"-"l"
}
}}

\newcommand{\XYexcc}[0]{\ensuremath{
\xygraph{!{<0pt,0pt>;<7pt,0pt>:<0pt,-7pt>::}
!{(-1,-1)}*{\scriptscriptstyle{\circ}}="c"
!{(-3,-3)}="e"
!{(1,-3)}*{\scriptscriptstyle{\circ}}="f"
!{(-0.4,-5)}*{\scriptscriptstyle{\circ}}="i"
!{(2.4,-5)}*{\scriptscriptstyle{\circ}}="j"
!{(-1.4,-7)}="m"
!{(0.4,-7)}="n"
!{(1.4,-7)}="o"
!{(3.4,-7)}="p"
!{(-3,-4)}*{\scriptscriptstyle 1}
!{(-1.4,-8)}*{\scriptscriptstyle 2}
!{(0.4,-8)}*{\scriptscriptstyle 5}
!{(1.4,-8)}*{\scriptscriptstyle 3}
!{(3.4,-8)}*{\scriptscriptstyle 4}
"f"-@[|(3)]"i"
"f"-@[|(3)]"j"
"i"-@[|(3)]"m"
"i"-@[|(3)]"n"
"j"-"o"
"j"-"p"
"c"-"e"
"c"-"f"
}
}}

\newcommand{\XYexcd}[0]{\ensuremath{
\xygraph{!{<0pt,0pt>;<7pt,0pt>:<0pt,-7pt>::}
!{(-1,-1)}*{\scriptscriptstyle{\circ}}="c"
!{(-3,-3)}*{\scriptscriptstyle{\circ}}="e"
!{(1,-3)}*{\scriptscriptstyle{\circ}}="f"
!{(-4.4,-5)}*{\scriptscriptstyle{\circ}}="g"
!{(-1.9,-5)}="h"
!{(-0.4,-5)}*{\scriptscriptstyle{\circ}}="i"
!{(2.4,-5)}*{\scriptscriptstyle{\circ}}="j"
!{(-5.4,-7)}="k"
!{(-3.4,-7)}="l"
!{(-1.4,-7)}="m"
!{(0.4,-7)}="n"
!{(1.4,-7)}="o"
!{(3.4,-7)}="p"
!{(-5.4,-8)}*{\scriptscriptstyle 1}
!{(-3.4,-8)}*{\scriptscriptstyle 5}
!{(-1.4,-8)}*{\scriptscriptstyle 2}
!{(0.4,-8)}*{\scriptscriptstyle 7}
!{(1.4,-8)}*{\scriptscriptstyle 4}
!{(3.4,-8)}*{\scriptscriptstyle 6}
!{(-1.9,-6)}*{\scriptscriptstyle 3}
"f"-@[|(3)]"i"
"f"-@[|(3)]"j"
"i"-@[|(3)]"m"
"i"-@[|(3)]"n"
"j"-"o"
"j"-"p"
"c"-"e"
"c"-"f"
"e"-@[|(3)]"g"
"e"-@[|(3)]"h"
"g"-@[|(3)]"k"
"g"-@[|(3)]"l"
}
}}

\DeclareMathOperator{\id}{id}
\DeclareMathOperator{\lt}{lt}

\DeclareMathOperator{\AntiCom}{AntiCom}
\DeclareMathOperator{\End}{End}

\DeclareMathOperator{\Hycom}{Hycom}
\DeclareMathOperator{\Grav}{Grav}
\DeclareMathOperator{\Ord}{Ord}
\DeclareMathOperator{\Vect}{Vect}

\theoremstyle{plain}

\newtheorem{theorem}{Theorem}[section]
\newtheorem{lemma}[theorem]{Lemma}
\newtheorem{corollary}[theorem]{Corollary}
\newtheorem{proposition}[theorem]{Proposition}

\theoremstyle{definition}

\newtheorem{definition}[theorem]{Definition}

\theoremstyle{remark}

\newtheorem{remark}[theorem]{Remark}
\newtheorem{example}[theorem]{Example}

\def\YEAR{\year}\newcount\VOL\VOL=\YEAR\advance\VOL by-1995
\def\firstpage{1}\def\lastpage{1000}
\def\received{}\def\revised{}
\def\communicated{}

\makeatletter
\def\magnification{\afterassignment\m@g\count@}
\def\m@g{\mag=\count@\hsize6.5truein\vsize8.9truein\dimen\footins8truein}
\makeatother

\font\eightrm=cmr8
\font\caps=cmcsc10                    
\font\Caps=cmcsc10 scaled \magstep1   

\makeatletter
\@specialpagefalse\headheight=8.5pt
\def\DocMath{}
\renewcommand{\@evenhead}{%
    \ifnum\thepage>\lastpage\rlap{\thepage}\hfill%
    \else\rlap{\thepage}\slshape\leftmark\hfill{\caps\SAuthor}\hfill\fi}%
\renewcommand{\@oddhead}{%
    \ifnum\thepage=\firstpage{\DocMath\hfill\llap{\thepage}}%
    \else{\slshape\rightmark}\hfill{\caps\STitle}\hfill\llap{\thepage}\fi}%
\makeatother

\def\TSkip{\bigskip}
\newbox\TheTitle{\obeylines\gdef\GetTitle #1
\ShortTitle  #2
\SubTitle    #3
\Author      #4
\ShortAuthor #5
\EndTitle
{\setbox\TheTitle=\vbox{\baselineskip=20pt\let\par=\cr\obeylines%
\halign{\centerline{\Caps##}\cr\noalign{\medskip}\cr#1\cr}}%
	\copy\TheTitle\TSkip\TSkip%
\def\next{#2}\ifx\next\empty\gdef\STitle{#1}\else\gdef\STitle{#2}\fi%
\def\next{#3}\ifx\next\empty%
    \else\setbox\TheTitle=\vbox{\baselineskip=20pt\let\par=\cr\obeylines%
    \halign{\centerline{\caps##} #3\cr}}\copy\TheTitle\TSkip\TSkip\fi%
\centerline{\caps #4}\TSkip\TSkip%
\def\next{#5}\ifx\next\empty\gdef\SAuthor{#4}\else\gdef\SAuthor{#5}\fi%
\ifx\received\empty\relax
    \else\centerline{\eightrm Received: \received}\fi%
\ifx\revised\empty\TSkip%
    \else\centerline{\eightrm Revised: \revised}\TSkip\fi%
\ifx\communicated\empty\relax
    \else\centerline{\eightrm Communicated by \communicated}\fi\TSkip\TSkip%
\catcode'015=5}}\def\Title{\obeylines\GetTitle}
\def\Abstract{\begingroup\narrower
    \parskip=\medskipamount\parindent=0pt{\caps Abstract. }}
\def\EndAbstract{\par\endgroup\TSkip}

\long\def\MSC#1\EndMSC{\def\arg{#1}\ifx\arg\empty\relax\else
     {\par\narrower\noindent%
     2010 Mathematics Subject Classification: #1\par}\fi}

\long\def\KEY#1\EndKEY{\def\arg{#1}\ifx\arg\empty\relax\else
	{\par\narrower\noindent Keywords and Phrases: #1\par}\fi\TSkip}

\newbox\TheAdd\def\Addresses{\vfill\copy\TheAdd\vfill
    \ifodd\number\lastpage\vfill\eject\phantom{.}\vfill\eject\fi}
{\obeylines\gdef\GetAddress #1
\Address #2 
\Address #3
\Address #4
\EndAddress
{\def\xs{4.3truecm}\parindent=0pt
\setbox0=\vtop{{\obeylines\hsize=\xs#1\par}}\def\next{#2}
\ifx\next\empty 
     \setbox\TheAdd=\hbox to\hsize{\hfill\copy0\hfill}
\else\setbox1=\vtop{{\obeylines\hsize=\xs#2\par}}\def\next{#3}
\ifx\next\empty 
     \setbox\TheAdd=\hbox to\hsize{\hfill\copy0\hfill\copy1\hfill}
\else\setbox2=\vtop{{\obeylines\hsize=\xs#3\par}}\def\next{#4}
\ifx\next\empty\ 
     \setbox\TheAdd=\vtop{\hbox to\hsize{\hfill\copy0\hfill\copy1\hfill}
                \vskip20pt\hbox to\hsize{\hfill\copy2\hfill}}
\else\setbox3=\vtop{{\obeylines\hsize=\xs#4\par}}
     \setbox\TheAdd=\vtop{\hbox to\hsize{\hfill\copy0\hfill\copy1\hfill}
	        \vskip20pt\hbox to\hsize{\hfill\copy2\hfill\copy3\hfill}}
\fi\fi\fi\catcode'015=5}}\gdef\Address{\obeylines\GetAddress}

\begin{document}
\Title
Quillen homology for operads via Gr\"obner bases
\ShortTitle 
\SubTitle   
\Author 
Vladimir Dotsenko and Anton Khoroshkin\footnote{The first author's research was supported by Grant GeoAlgPhys 2011-2013 awarded by the University of Luxembourg and by an IRCSET research fellowship. The second author's research was supported by grants NSh3349.2012.2, RFBR-10-01-00836, RFBR-CNRS-10-01-93111, RFBR-CNRS-10-01-93113, and a grant Ministry of Education and Science of the Russian Federation under the contract 14.740.11.081.}
\ShortAuthor 
Vladimir Dotsenko and Anton Khoroshkin
\EndTitle
\Abstract 
The main goal of this paper is to present a way to compute Quillen homology of operads. The key idea is to use the notion of a shuffle operad we introduced earlier; this allows to compute, for a symmetric operad, the homology classes and the shape of the differential in its minimal model, although does not give an insight on the symmetric groups action on the homology. Our approach goes in several steps. First, we regard our symmetric operad as a shuffle operad, which allows to compute its Gr\"obner basis. Next, we define a combinatorial resolution for the ``monomial replacement'' of each shuffle operad (provided by the Gr\"obner bases theory). Finally, we explain how to ``deform'' the differential to handle every operad with a Gr\"obner basis, and find explicit representatives of Quillen homology classes for a large class of operads. We also present various applications, including a new proof of Hoffbeck's PBW criterion, a proof of Koszulness for a class of operads coming from commutative algebras, and a 
homology computation for the operads of Batalin--Vilkovisky algebras and of Rota--Baxter algebras.

\EndAbstract
\MSC 
\EndMSC
\KEY 
\EndKEY
\Address 
School of Mathematics, Trinity College, Dublin 2, Ireland
\Address
Simons Center for Geometry and Physics, Stony Brook University, Stony Brook, NY 11794-3636, USA 
and 
ITEP, Bolshaya Cheremushkinskaya 25, 117259, Moscow, Russia
\Address
\Address
\EndAddress

\section*{Introduction}

\subsection*{Context and the main goal of the paper}
Quillen's philosophy of homotopical algebra \cite{QuillenHomotopical,QuillenComAlg} suggests to study invariants of associative algebras (and their variants in various monoidal categories) within the homotopical category obtained from the usual category of (differential graded) algebras via the localisation inverting all quasi-isomorphisms. One of the key invariants of that sort is what is now called Quillen homology, the left derived functor of the functor of indecomposables $A\mapsto A^{ab}:=A/A^2$ (this functor is also somewhat informally called ``abelianisation''). Since we work in the homotopical category, studying an algebra is the same as studying its cofibrant replacement, which in many known model categories (associative algebras, operads etc.) is given by a (quasi-)free resolution. One such resolution, given by the cobar-bar construction, is readily available, however, sometimes it is preferable (and possible) to have a smaller resolution. The so called minimal resolution (if exists) has Quillen 
homology of the algebra as its space of generators. In general, given an algebra~$A$ and a free resolution~$(\calF_\bullet,d)$ of~$A$, Quillen homology $H^Q(A)$ is isomorphic to the homology of the differential induced on the space of indecomposable elements of $\calF_\bullet$, i.e., on the space of generators of our resolution. As a vector space, it can also be identified with the homology of the differential induced on the space of generators of a resolution of the trivial~$A$-module by free right $A$-modules, that is the appropriate $\mathrm{Tor}$ groups, also called syzygies of the given algebra. One general way of computing the space of syzygies is a step-by-step procedure which is usually referred to as the Koszul--Tate method~\cite{Koszul50,Tate57}. However, in some cases it is possible to visualise the whole space of syzygies ``in one go'', like the Koszul duality theory~\cite{GK,Priddy} suggests. A question raised by Jean-Louis Loday in~\cite[Question 7]{Loday2012} is to compare the computations via 
the Koszul duality theory (when available) with those by the Koszul--Tate approach. In this paper, we give a method that brings those two approaches together, applying the machinery of Gr\"obner bases and thus understanding the intrinsic structure of relations between relations in the spirit of Koszul duality.

One of most important practical results provided (in many different frameworks) by Gr\"obner bases is that when dealing with various linear algebra information (bases, dimensions etc.) one can replace an algebra with complicated relations by an algebra with monomial relations without losing any information of that sort. When it comes to questions of homological algebra, things become more subtle, since homology may ``jump up'' for a monomial replacement of an algebra. However, the idea of applying Gr\"obner bases to problems of homological algebra is far from hopeless. It turns out that for monoids with monomial relations it is often possible to construct very neat resolutions that can be used for various computations; furthermore, the data computed by these resolutions can be used to obtain results in the general (not necessarily monomial) case. The main goal of this paper is to explain this approach in detail for computations of Quillen homology for operads.

\subsection*{Proposed methods}

Operads that usually arise naturally in various topics are \emph{symmetric} operads; they encode intrinsic properties of operations with several arguments acting on certain algebraic or geometric objects, and as such are equipped with the action of symmetric groups. The presence of such an action instantly implies that there is no meaningful notion of a monomial operad: any notion of that sort is not rich enough for every operad to have a monomial replacement. A way to deal with this problem proposed in~\cite{DK} is to include symmetric operads in a larger universe of \emph{shuffle} operads, where every object has a monomial replacement. On shuffle operads, no symmetries are allowed to act directly: the only way symmetries enter the game is through the formulae for compositions of operations. In that category, it is possible to make the Gr\"obner bases machinery work, and hence there is hope that it can be applied to questions of homological algebra. It is indeed possible, along the following 
lines. We begin with a resolution which generally not minimal even in the monomial case, but has the advantage of being purely combinatorial and not using much information about the underlying monoidal category. It is based on the inclusion--exclusion principle, and is in a sense a version of the cluster method of enumerative combinatorics due to Goulden and Jackson~\cite{GJ}. The resolution obtained is not always minimal, and we also discuss how to use it to compute the Quillen homology as a vector space, using algebraic Morse theory. This is followed, by an explanation of how to ``deform'' the differential of our resolution to incorporate lower terms of relations and handle arbitrary algebras with known Gr\"obner bases.  Note that the Quillen homology of a symmetric operad is, as a homotopy shuffle co-operad, isomorphic to the Quillen homology of that operad considered in the shuffle category. Since our results allow to compute the homotopy co-operad structure on the Quillen homology via the homotopy 
transfer theorem for homotopy co-operads~\cite{DCV}, in principle we recover most of the information on Quillen homology in the symmetric category (and of course the shape the differential in the minimal model). The reader will see that in some of the examples we discuss.	  

As we mentioned earlier, our approach generalises the Koszul duality theory for defining relations of arbitrary degrees. In the case of a Koszul operad, one is able to write down a formula for its minimal resolution right away: such a resolution has the Koszul dual co-operad as its space of generators. Most known examples of Koszul algebras and operads actually satisfy the PBW condition~\cite{Hoffbeck,Priddy}, or equivalently have a quadratic Gr\"obner basis~\cite{DK,PP}. In that case, the minimal resolution provided by the Koszul duality theory coincides with that obtained by our methods. Generally, Gr\"obner bases allow to choose a ``good'' system of relations that captures the structure of relations between relations (higher syzygies). 

\subsection*{Related results}
In the case of usual associative algebras and right modules, the approach we discuss has been known since the celebrated paper of Anick~\cite{Anick} where for a monomial algebra a minimal right module resolution of the trivial module was computed, and an explicit way to deform the differential was presented to handle the general case. Later, Anick's resolution was generalised to the case of categories by Malbos~\cite{Malbos} who also asked whether this work could be extended to the case of operads.  Results of this paper give such an extension, and suggest a way to handle associative algebras presented via generators and relations in many different monoidal categories (e.g. commutative associative algebras, associative dialgebras, (shuffle) coloured operads, dioperads, $\frac12$PROPs) in a uniform way. If, in addition, we assume that our algebras are linear spans of algebras enriched in sets, our constructions are closely related to those of free polygraphic resolutions for $(\infty,1)$-categories obtained 
by methods of 
rewriting theory~\cite{Met1}; this relationship is currently investigated by the first author in a joint work with Yves Guiraud and Philippe Malbos, and will be discussed elsewhere.

\subsection*{Overview of applications}
There are various applications of our approach; some of them are presented in this paper. Two interesting theoretical applications are a new short proof of Hoffbeck's PBW criterion for operads \cite{Hoffbeck}, and an upper bound on the homology for operads obtained from commutative algebras; in particular, we prove that an operad obtained from a Koszul commutative algebra is Koszul. Some interesting concrete examples where all steps of our construction can be completed are the case of the operad $RB$ of Rota--Baxter algebras, its noncommutative analogue~$ncRB$, and, the last but not the least, the operad~$BV$ of Batalin--Vilkovisky algebras. Using our methods, we were able to compute Quillen homology of the operad~$BV$ and relate it to the gravity operad of Getzler~\cite{Getzler1}. After completing the first version of this paper in 2009~\cite{DKres2009}, we learned that these (and other) results concerning the operad~$BV$ were also obtained independently by Drummond-Cole and Vallette~\cite{DCV}. Our methods 
appear to be completely different; we also believe that our approach to the operad~$BV$ is of independent interest as an illustration of a rather general method to compute Quillen homology. It also shows how to use information coming ``from the symmetric world'' to partly understand the shape of trees that appear in the formula for the differential of the minimal model. This idea (to move to the universe of shuffle operads, compute the vector space structure there, and then to use known information on our operad to obtain results about the symmetries of homology) also belongs to the core of the shuffle operad approach. 

\subsection*{Plan of the paper}
This paper is organised as follows. 

In Section~\ref{sec:operads}, we recall necessary background information on shuffle operads, and provide references for definitions and results that are relevant for the paper but are not discussed in detail. 

In Section~\ref{oper_monom}, we present an ``inclusion--exclusion'' resolution for an arbitrary shuffle operad with monomial relations, which we then use in the subsequent sections. 

In Section \ref{homology_classes}, we use algebraic Morse theory to construct representatives for Quillen homology classes under a minor assumption on the combinatorics of defining relations. This section includes brief recollections of algebraic Morse theory, as well as full proofs of existence of Morse matchings; whereas results of that section are important for some of our examples, a reader primarily interested in applications may skip the proofs without any disadvantages for understanding the rest of the paper.

In Section~\ref{resol_general}, we use a version of homological perturbation to obtain a resolution for a general shuffle operad with a Gr\"obner basis. 

In Section~\ref{appl}, we exhibit applications of our results outlined above. Those are a new proof of the PBW criterion, homology estimates for operads coming from commutative algebras, and a computation of Quillen homology for the operads~$RB$, $ncRB$, and~$BV$.

\subsection*{Acknowledgments} Work on this paper ended up being spread over several years, and we benefited a lot from discussing preliminary versions of this paper with our colleagues. We thank Fr\'ed\'eric Chapoton, Yves Guiraud, Eric Hoffbeck, Muriel Livernet, Jean--Louis Loday, Philippe Malbos and Dmitri Piontkovski for useful discussions, remarks and references to literature. Special thanks are due to Li Guo for drawing our attention to Rota--Baxter algebras as an example to which our methods could be applied, and especially to Bruno Vallette for many stimulating discussions and in particular for explaining the approach to the operad~$BV$ used in his joint work with Gabriel Drummond-Cole. The first author gratefully acknowledges the hospitality of the University Lyon 1 during the final stage of working on this paper. The second author wishes to thank Dublin Institute for Advanced Studies and the University of Luxembourg, where parts of this paper were completed, for invitations to visit and excellent 
working conditions. 

\section{Recollections}\label{sec:operads}

All vector spaces and (co)chain complexes throughout this work are defined over an arbitrary field~$\k$. To handle suspensions, we introduce a formal symbol~$s$ of degree~$1$, and define, for a graded vector space~$V$, its suspension $sV$ as $\k s\otimes V$. All algebras and operads are assumed non-unital; to adapt our results for unital algebras, one has to restrict the setup to augmented algebras, and consider the abelianisation $A^{ab}=A_+/(A_+)^2$, where $A_+$ is the augmentation ideal.

The main thing about shuffle operads that is crucial for our constructions is the relevant combinatorics of trees. Hence, in this section we pay most attention in explaining that combinatorics in detail. We also discuss the precise relationship between homological/homotopical results obtained in the symmetric and in the shuffle category, since most of the applications we have in mind concern symmetric operads. For information on operads in general, we refer the reader to the book~\cite{LV}, for information on shuffle operads and Gr\"obner bases for operads in not necessarily quadratic case --- to our paper~\cite{DK}. Throughout this paper by an operad, unless otherwise specified, we mean a shuffle operad: there is no machinery of Gr\"obner bases available in the symmetric case, so we have to sacrifice the symmetric groups action. Of course, in some cases we deal with non-symmetric operads, and in that case there is nothing to sacrifice, and in fact the story is somewhat richer since one can include constants 
(operations of arity~$0$) in the picture and avail of Gr\"obner bases at no additional cost. For details on that, see~\cite{DV}. For the sake of brevity, in this paper we shall discuss shuffle operads in 
detail, while non-symmetric operads will only appear in some examples.

\subsection{Shuffle operads, bar construction and homology}

Let us denote by~$\Ord$ the category whose objects are non-empty finite ordered sets (with order-preserving bijections as morphisms). Also, we denote by $\Vect$ the category of vector spaces (with linear operators as morphisms). It is usually enough to assume vector spaces to be finite-dimensional, though sometimes more generality is needed, and one assumes, for instance, that they are graded with finite-dimensional homogeneous components. 

\begin{definition}
A \emph{(non-symmetric) collection} is a contravariant functor from the category~$\Ord$ to the category~$\Vect$. Because of functoriality, a nonsymmetric collection $\calP$ is completely determined by its \emph{components} $\calP(n):=\calP(\{1,\ldots,n\})$, $n\ge 1$.

For two nonsymmetric collections $\calP$ and $\calQ$, the \emph{shuffle composition product} of $\calP$ and $\calQ$ is the non-symmetric collection $\calP\circ_{sh}\calQ$ defined by the formula
\[
(\calP\circ_{sh}\calQ)(I):=\bigoplus_{k\ge 1}\calP(k)\otimes\left(\bigoplus_{\phi\colon I\twoheadrightarrow\{1,\ldots,k\}}\calQ(\phi^{-1}(1))\otimes\cdots\otimes\calQ(\phi^{-1}(k))\right), 
\]
where the sum is taken over all \emph{shuffling surjections}~$f$, that is surjections for which~$\min\phi^{-1}(i)<\min\phi^{-1}(j)$ whenever~$i<j$. 
\end{definition}

\begin{proposition}[\cite{DK}]
The shuffle composition product equips the category of non-symmetric collections with a structure of a monoidal category.
\end{proposition}

\begin{definition}
A \emph{shuffle operad} is a monoid in the category of non-symmetric collections equipped with the shuffle composition product.
\end{definition}

For the monoidal category of shuffle operads, it is possible to define the bar complex of an operad~$\calO$. The \emph{bar complex} $\mathbf{B}^\bullet(\calO)$ is a dg co-operad freely generated by the suspension $s\calO$; the differential comes from operadic compositions in~$\calO$. Similarly, for a co-operad $\calQ$, it is possible to define the \emph{cobar complex}~$\Omega^\bullet(\calQ)$, which is a dg operad freely generated by~$s^{-1}\calQ$, with the appropriate differential. The bar-cobar construction $\Omega^\bullet(\mathbf{B}^\bullet(\calO))$ gives a free resolution of~$\calO$. This can be proved in a rather standard way, similarly to known proofs in the case of operads, properads etc.~\cite{F,GK,ValPROP}. The general homotopical algebra philosophy mentioned in the introduction is applicable in the case of operads as well; various checks and justifications needed to ensure that are quite standard and similar to the ones available in the literature; we refer the reader to \cite{BM, Fresse, Harper1, 
MarklModel, Rezk, Spitzweck} where symmetric operads are handled. Thus, the Quillen homology of an operad can be computed as homology of its bar complex (since the abelianisation of the bar-cobar construction is the bar complex), though sometimes this complex is too big to handle, so it is important to seek more economic free resolutions. Our approach allows to build free resolutions for shuffle operads with known Gr\"obner bases, thus giving an alternative way to compute Quillen homology. 

\subsection{Symmetric vs shuffle}

Let us explain precisely what information on Quillen homology for symmetric operads ``survives'' in the shuffle world, and what is lost. Of course, the information on the symmetric group actions does get lost. However, we argue that all other relevant structures on the homology do survive. Recall that the forgetful functor ${}^f\colon\calP\to\calP^f$ from the category of symmetric collections to the category of nonsymmetric collections (with the shuffle product) is monoidal~\cite[Prop.~3]{DK}. This easily implies the following 
\begin{proposition}
For a symmetric operad~$\calP$, we have 
\[
\mathbf{B}^\bullet(\calP)^f\simeq\mathbf{B}^\bullet(\calP^f), 
\]
that is the (symmetric) bar complex of $\calP$ is isomorphic, as a shuffle dg co-operad, to the (shuffle) bar complex of~$\calP^f$. 
\end{proposition}

Appropriate homotopy transfer for homotopy co-operads \cite{DCV} (together with the observation that homotopy co-operad maps on the Quillen homology are up to suspension equal to components of the differential in the minimal model) implies 
\begin{corollary}
For a symmetric operad~$\calP$, we have 
\[
H^Q(\calP)^f\simeq H^Q(\calP^f), 
\]
that is the (symmetric) Quillen homology $\calP$ is isomorphic, as a shuffle homotopy co-operad, to the (shuffle) Quillen homology of~$\calP^f$. Also,
if $\calR_\calO$ denotes the minimal model of an operad $\calO$ in the appropriate category (symmetric or shuffle),, we have
 \[
(\calR_\calP)^f\simeq \calR_{\calP^f}
 \]
as shuffle dg operads.
\end{corollary}

In particular, this means that on the minimal model in the shuffle category it is in principle possible to introduce a symmetric groups action compatible with the differential so that it becomes precisely the minimal model in the symmetric category.

\subsection{Tree monomials}

Let us recall tree combinatorics used to describe monomials in shuffle operads. See~\cite{DK} for more details.

Basis elements of the free operad are represented by (decorated) trees. A (rooted) \emph{tree} is a non-empty connected directed graph $T$ of genus~$0$ for which each vertex has at least one incoming edge and exactly one outgoing edge. Some edges of a tree might be bounded by a vertex at one end only. Such edges are called \emph{external}. Each tree should have exactly one outgoing external edge, its \emph{output}. The endpoint of this edge which is a vertex of our tree is called the \emph{root} of the tree. The endpoints of incoming external edges which are not vertices of our tree are called \emph{leaves}.

Each tree with~$n$ leaves should be (bijectively) labelled by the standard $n$-element set~$[n]=\{1,2,\ldots,n\}$. For each vertex $v$ of a tree, the edges going in and out of $v$ will be referred to as inputs and outputs at~$v$. A tree with a single vertex is called a \emph{corolla}. There is also a tree with a single input and no vertices called the \emph{degenerate} tree. Trees are originally considered as abstract graphs but to work with them we would need some particular representatives. For a tree with labelled leaves, its \emph{canonical planar representative} is defined as follows. In general, an embedding of a (rooted) tree in the plane is determined by an ordering of inputs for each vertex. To compare two inputs of a vertex~$v$, we find the minimal leaves that one can reach from~$v$ via the corresponding input. The input for which the minimal leaf is smaller is considered to be less than the other one. 

Let us introduce an explicit realisation of the free operad generated by a collection $\calM$. The basis of this operad will be indexed by canonical planar representatives of trees with decorations of all vertices. First of all, the simplest possible tree is the degenerate tree; it corresponds to the unit of our operad. The second simplest type of trees is given by corollas. We shall fix a basis of~$\calM$ and decorate the vertex of each corolla with a basis element; for a corolla with $n$ inputs, the corresponding element should belong to the basis of~$\calV(n)$. The basis for whole free operad consists of all canonical planar representatives of trees built from these corollas 
(explicitly, one starts with this collection of corollas, defines compositions of trees in terms of grafting, and then considers all trees obtained from corollas by iterated shuffle compositions). We shall refer to elements of this basis as \emph{tree monomials}. Vice versa, if we forget the labels of vertices and leaves of a tree monomial~$\alpha\in\calF_\calM$, we obtain a planar tree. We shall refer to this planar tree as \emph{the underlying tree of~$\alpha$}. 

For example, if $\calO=\calF_\calM$ is the free operad for which the component $\calM(n)$ is only non-zero for $n=2$, and $\calM(2)=\k\{\circ\}$, the basis of $\calF_\calM(3)$ is given by the tree monomials
 \[
\LYY{1}{2}{3}{}{}, \quad  \LYY{1}{3}{2}{}{},\quad \text{and} \quad \RYY{2}{3}{1}{}{}\ .
 \]

There are two standard ways to think of elements of an operad defined by generators and relations: using either tree monomials or operations.  For example, the above tree monomials correspond to operations
 \[
(a_1\circ a_2)\circ a_3, \quad (a_1\circ a_3)\circ a_2,\quad \text{and} \quad a_1\circ (a_2\circ a_3)\ .
 \]
Our approach is somewhere in the middle between the two viewpoints: we strongly encourage the reader to think of tree monomials, but to write down the formulas required for definitions and proofs we prefer the language of operations since it makes things more compact. 

\begin{example}\label{ex:tree}
The following is a tree monomial in the free operad $\calF_\calM$ generated by some collection $\calM$ with $\circ,\bullet\in\calM(2)$:
 \[
\XYex\ .
 \]
In the language of operations, it corresponds to the operation \[((a_1\circ a_3)\bullet a_5)\circ ((a_2\bullet a_7)\circ (a_4\circ a_6))\ .\]
\end{example}

Divisors of a tree monomial $\alpha$ in the free operad correspond to a special kind of subgraphs of its underlying tree. Allowed subgraphs contain, together with each vertex, all its incoming and outgoing edges (but not necessarily other endpoints of these edges). Throughout this paper we consider only this kind of subgraphs, and we refer to them as subtrees hoping that it does not lead to any confusion. Clearly, a subtree~$T'$ of every tree~$T$ is a tree itself. Let us define the tree monomial $\alpha'$ corresponding to~$T'$. To label vertices of~$T'$, we recall the labels of its vertices in~$\alpha$.  We immediately observe that these labels match the restriction labels of a tree monomial should have: each vertex has the same number of inputs as it had in the original tree, so for a vertex with $n$ inputs its label does belong to the basis of~$\calM(n)$. To label leaves of~$T'$, note that each such leaf is either a leaf of~$T$, or is an output of some vertex of~$T$. This allows us to assign to each 
leaf~$l'$ of~$T'$ a leaf~$l$ of~$T$, which we call the smallest descendant of~$l'$: if $l'$ is a leaf of~$T$, put $l=l'$, otherwise let~$l$ be the smallest leaf of $T$ that can be reached through~$l'$. We then number the leaves according to their smallest descendants: the leaf with the smallest 
possible descendant gets the label~$1$, the second smallest~--- the label~$2$ etc. 

\begin{example}\label{ex:pattern}
Let us consider the following two choices of subtrees of the tree from Example~\ref{ex:tree}:
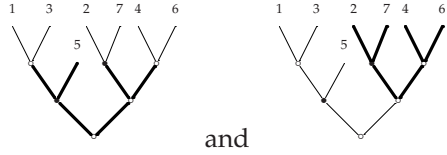
 \[
\XYexa\quad \text{and} \quad \XYexb\ .
 \]
In the first case, the subtree marked by bold lines yields the tree monomial $\LYRYAB{1}{5}{2}{4}$, and the "standardisation" re-labelling, as above, gives the tree monomial $\LYRYAB{1}{4}{2}{3}$. 
In the second case, the subtree marked by bold lines yields the tree monomial $\LYRYAB{2}{7}{4}{6}$, and the "standardisation" re-labelling gives the same tree monomial $\LYRYAB{1}{4}{2}{3}$. 
In the language of operations, \[(a_1\bullet a_5)\circ (a_2\circ a_4)\simeq (a_2\bullet a_7)\circ (a_4\circ a_6)\simeq (a_1\bullet a_4)\circ(a_2\circ a_3).\] Thus, $\LYRYAB{1}{4}{2}{3}$ occurs as a divisor of the original tree monomial at two different places.
\end{example}

For two tree monomials $\alpha$, $\beta$ in the free operad $\calF_\calM$, we say that \emph{$\alpha$ is divisible by $\beta$}, if there exists a subtree of the underlying tree of $\alpha$ for which the corresponding tree monomial $\alpha'$ is equal to~$\beta$.

There exist several ways to introduce a total ordering of tree monomials in such a way that the operadic compositions are compatible with that total ordering. A Gr\"obner basis of an ideal $I$ of the free operad is a system $S$ of generators of~$I$ for which the leading monomial of every element of the ideal is divisible by one of the leading terms of elements of~$S$. Such a system of generators allows to perform ``long division'' modulo~$I$, computing for every element its canonical representative. There exists an algorithmic way to compute a Gr\"obner basis starting from any given system of generators (``Buchberger's algorithm for shuffle operads''). For our purposes, it is important to note that if the tree monomials of our operad have additional internal grading, and the relations are homogeneous with respect to that grading, then the corresponding reduced Gr\"obner basis is also homogeneous, as well as all our homological constructions.

\subsection{Operads in the differential graded setting} 

The above description of the free shuffle operad works almost literally when we work with operads whose components are chain complexes (as opposed to vector spaces), and the symmetric monoidal structure on the corresponding category involves signs. The only difference is that every tree monomial should carry an ordering of its internal vertices, so that two different orderings contribute appropriate signs. In this section, we give an example of a shuffle dg operad that should help a reader to understand the graded case better; this operad was introduced and explored in~\cite{MR}.

\begin{definition}
The \emph{odd $(2k+1)$-associative operad} is a non-symmetric operad with one generator~$\mu$ of arity~$2k+1$ and odd homological degree, and relations 
 \[
\mu\circ_p\mu=\mu\circ_{2k+1}\mu \quad \text{for all} \quad p\le 2k.   
 \]
\end{definition}

Let us show that the Buchberger algorithm for operads from~\cite{DK} discovers a cubic relation in the Gr\"obner basis for this operad, thus showing that this operad fails to be PBW in the sense of Hoffbeck~\cite{Hoffbeck} (for this particular ordering). We use the path-lexicographic ordering of monomials. 

From the common multiple $(\mu\circ_1\mu)\circ_1\mu$ of the leading term $\mu\circ_1\mu$ with itself, we compute the S-polynomial
\[
(\mu\circ_{2k+1}\mu)\circ_1\mu-\mu\circ_1(\mu\circ_{2k+1}\mu). 
\]
We can perform the following chain of reductions (with leading monomials underlined):
\begin{multline*}
(\mu\circ_{2k+1}\mu)\circ_1\mu-\underline{\mu\circ_1(\mu\circ_{2k+1}\mu)}=
(\mu\circ_{2k+1}\mu)\circ_1\mu-\underline{(\mu\circ_1\mu)\circ_{2k+1}\mu}\mapsto\\ \mapsto
\underline{(\mu\circ_{2k+1}\mu)\circ_1\mu}-(\mu\circ_{2k+1}\mu)\circ_{2k+1}\mu=
-\underline{(\mu\circ_1\mu)\circ_{4k+1}\mu}-(\mu\circ_{2k+1}\mu)\circ_{2k+1}\mu\mapsto\\ \mapsto
-(\mu\circ_{2k+1}\mu)\circ_{4k+1}\mu-\underline{(\mu\circ_{2k+1}\mu)\circ_{2k+1}\mu}=
-(\mu\circ_{2k+1}\mu)\circ_{4k+1}\mu-\underline{\mu\circ_{2k+1}(\mu\circ_1\mu)}\mapsto\\ \mapsto
-(\mu\circ_{2k+1}\mu)\circ_{4k+1}\mu-\mu\circ_{2k+1}(\mu\circ_{2k+1}\mu)=
-2(\mu\circ_{2k+1}\mu)\circ_{4k+1}\mu.
\end{multline*}
Note that we used the formula $(\mu\circ_{2k+1}\mu)\circ_1\mu=-(\mu\circ_1\mu)\circ_{4k+1}\mu$ which reflects the fact that the operation $\mu$ is of odd homological degree.

The monomial $(\mu\circ_{2k+1}\mu)\circ_{4k+1}\mu$ cannot be reduced further, and we recover the relation $(\mu\circ_{2k+1}\mu)\circ_{4k+1}\mu=0$ discovered in~\cite{MR}. Furthermore, we arrive at the following proposition (note the similarity with the computation of the Gr\"obner basis for the operad~$\AntiCom$ in~\cite{DK}).

\begin{proposition}
Elements $\mu\circ_p\mu-\mu\circ_{2k+1}\mu$ with $1\le p\le 2k$ and $(\mu\circ_{2k+1}\mu)\circ_{4k+1}\mu$ form a Gr\"obner basis for the operad of odd $(2k+1)$-associative algebras.
\end{proposition}

\section{Resolution for monomial relations}\label{oper_monom}

Assume that the operad~$\calO=\calF_\calM/(\calG)$ is generated by a collection of finite sets $\calM=\{\calM(n)\}$, and that $\calG$ consists of tree monomials, so we are dealing with an operad that only has monomial relations. We shall explain how to construct a free resolution of~$\calO$. 

Our first step is to construct a free shuffle dg operad $\calA$ which does not take the account relations of $\calO$; it is a somewhat universal object for operads generated by~$\calM$, various suboperads of $\calA$ will be used as resolutions for various choices of $\calG$.

\subsection{The inclusion--exclusion operad}

Let $T$ be a tree monomial, and let the symbols $S_1,\ldots,S_q$ be in one-to-one correspondence with all the divisors of $T$. We denote by $\calA(T)$ the vector space $\k T\otimes\Lambda(S_1,\ldots,S_q)$. We shall say that \emph{underlying tree monomial} for elements of this vector space is~$T$. The degree~$-1$ derivations $\partial_i$ on the exterior algebra defined by the rule $\partial_i(S_j)=\delta_{ij}$ anticommute, and the differential $d=\sum_{i=1}^q\partial_i$ makes $\calA(T)$ into a chain complex isomorphic to the augmented chain complex of a $(q-1)$-dimensional simplex $\Delta^{q-1}$. 

By definition, the chain complex $\calA(n)$ is the direct sum of complexes $\calA(T)$ over all tree monomials $T$ with $n$ leaves. There is a natural operad structure on the collection $\calA=\{\calA(n)\}$; the operadic composition composes the trees, and computes the wedge product of symbols labelling their divisors. Overall, we defined a shuffle dg operad, which we shall call \emph{the inclusion--exclusion operad}. 

Let us emphasize that the symbols $S_{i_r}$ correspond to divisors, i.e. mark \emph{occurrences of tree monomials in $T$} rather than monomials themselves, so in particular the Koszul sign rule does not imply that a composition of an element of our operad with itself is equal to zero.  Basically, when computing products, the $S$-symbols ``remember'' which divisors of factors they come from. Graphically, it is convenient to think of basis elements of our chain complex as tree monomials with some of the occurrences of relations additionally marked, as in Example~\ref{ex:pattern}.

\smallskip

The following example should make our construction more clear.

\begin{example}
Assume that the operad $\calO$ has two binary generators~$\circ$ and~$\bullet$. Then the corresponding operad $\calA$ contains, among others, two elements
 $$
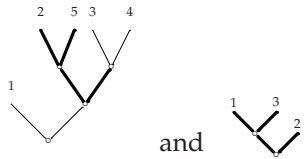
\XYexcc \quad \text{and}\quad \LYYb{1}{3}{2}{}\ .
 $$
An appropriate shuffle composition of these two produces the element
 $$
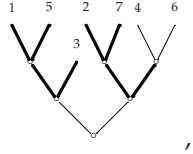
\XYexcd\ ,
 $$
where two different divisors are marked. Incidentally, the underlying tree monomial for each of them is $\LYY{1}{3}{2}{}$. Let us call the first tree $T$, the second tree $T'$, the first divisor $S$, and the second divisor $S'$. On the level of formulas, we have
\[
(T\otimes S)\circ_\phi(T'\otimes S')=(T\circ_\phi T')\otimes S_1\wedge S_2, 
\]
where $\phi\colon\{1,2,3,4,5,6,7\}\twoheadrightarrow\{1,2,3,4,5\}$ is the shuffle surjection with $\phi(1)=\phi(3)=\phi(5)=1$, $\phi(2)=2$, $\phi(4)=3$, $\phi(6)=4$, $\phi(7)=5$, and $S_1$ and $S_2$ indicate the two different divisors of $T\circ_\phi T'$ equal to~$\LYY{1}{3}{2}{}$. Note that even though the underlying tree monomials of the two divisors coincide, the wedge product is not equal to zero, since the letters $S$ correspond to distinct divisors, that is occurrences of tree monomials, not tree monomials themselves.
\end{example}

In fact, the underlying operad of the dg operad $\calA$ is free. Indeed, let us call a ``monomial'' $T\otimes S_{i_1}\wedge\cdots\wedge S_{i_q}$, $q\ge 0$, \emph{indecomposable}, if it is not a composition in the operad $\calA_\calG$ of two monomials of the same type. (This means that each edge between the two internal vertices of~$T$ is an edge between two internal vertices of at least one of the divisors $S_{i_1}$, \ldots, $S_{i_q}$; note that some of the internal vertices of $T$ are leaves of its divisors, and hence are not considered internal vertices of the respective divisors.) It is easy to see that $\calA$ is freely generated by indecomposable elements; those are elements $m\otimes 1$ with $m\in\calM$ being a generator of~$\calO$, and indecomposable monomials $T\otimes S_{i_1}\wedge\cdots\wedge S_q$, $q\ge1$. For each monomial $T\otimes S_{i_1}\wedge\cdots\wedge S_{i_q}$ different from the generators described above, and each internal edge of $T$ that is not an internal edge of either of $S_{i_1}$, \
ldots, $S_{i_q}$, the endpoint of that edge which is further from the root of $T$ is a ``grafting point'': the subtree growing from this vertex (and its divisors among $S_{i_j}$) factors out in our operad. This factorisation procedure gives a unique way to factorise elements as compositions of generators.

So far we did not use the relations of our operad. Let us incorporate relations in the picture. 

\subsection{Suboperads of the inclusion--exclusion operad}

Let $\calG$ be the set of relations of our operad~$\calO$. The dg operad $(\calA_\calG,d)$ is defined similarly to $\calA$, but with the additional restriction that every symbol $S_k$ corresponds to a divisor of $T$ for which the underlying tree monomial is a relation. The differential~$d$ is the restriction of the differential defined above. Informally, an element of the operad $\calA_\calG$ is a tree with some distinguished divisors that are relations from the given set.

\begin{theorem}
The dg operad $(\calA_\calG,d)$ is a free resolution (as a shuffle operad) of the corresponding operad with monomial relations $\calO=\calF_\calM/(\calG)$.
\end{theorem}

\begin{proof}
Similarly to the case of the operad $\calA$, the operad $\calA_\calG$ is freely generated by its elements $m\otimes 1$ with $m\in\calM$ and all indecomposable monomials $T\otimes S_1\wedge\cdots\wedge S_q$, $q\ge1$, where each of the divisors $S_i$ is a relation of~$\calO$.

Let us prove that $\calA_\calG$ provides a resolution for $\calO$. Since the differential $d$ only omits wedge factors but does not change the tree monomial, the chain complex $\calA_\calG$ is isomorphic to the direct sum of chain complexes $\calA_\calG^T$ spanned by the elements for which the first tensor factor is the given tree monomial~$T$. If $T$ is not divisible by any relation, the complex $\calA_\calG^T$ is concentrated in degree~$0$ and is spanned by~$T\otimes 1$. Thus, to prove the theorem, we should show that $\calA_\calG^T$ is acyclic whenever $T$ is divisible by some relation $g_i$.

Assume that there are exactly $k$ divisors of~$T$ which are relations of~$\calO$. We immediately see that the complex~$\calA_\calG^T$ is isomorphic to the chain complex of a simplex $\Delta^{k-1}$ which is acyclic whenever $k>0$.
\end{proof}

\begin{remark}
Using the machinery of twisting cochains~\cite{twisting}, one can obtain from the free dg operad resolution a resolution of the trivial module by free right modules whose spaces of generators of various homological degrees are the same as the spaces of generators of the original operad resolution. More precisely, the differential of every generator in our free operad resolution is a sum of compositions of generators; this provides the space of generators with a structure of a homotopy co-operad, and the twisting cochain method applies. See~\cite{Lefevre-Hasegawa,Proute} for details in the (simpler) case of associative algebras, and \cite{DCV} for details in the operad case. 
\end{remark}

\section{Homology classes for monomial operads}\label{homology_classes}

In general, the fact that ``trees grow in several different directions'', means that it is more difficult to describe representatives of homology classes combinatorially in the same way as it can be done for the case of associative algebras~\cite{Anick}. However, in some cases it is possible to come up with a reasonable description. In this section, we shall describe homology classes under a minor restriction on the combinatorics of defining relations. 

\subsection{Algebraic Morse theory: recollections}
To obtain our description, we use the algebraic Morse theory developed independently in~\cite{AlgMorse, Kozlov, Skoldberg}. We refer the reader to those references for details; for our purposes, the algebraic Morse theory is a way to describe a smaller subcomplex of a chain complex having the same homology. It is done as follows. Suppose that our chain complex $(C_\bullet,d)$ has a basis $X=\sqcup_{i\ge1} X_i$, where $X_i$ is the basis of the space $C_i$ of our complex; this basis should be finite or satisfy some local finiteness condition (i.e., internal grading). We consider a directed graph $\Gamma$ whose vertex set $V$ coincides with $X$, and the edge set $E$ reflects the combinatorics of the differential: there is an edge from $x\in X_i$ to $y\in X_{i-1}$ if $y$ appears in $d(x)$ with a non-zero coefficient. A set of edges $M\subset E$ is called a \emph{Morse matching} if two conditions are satisfied:
\begin{enumerate}
 \item every vertex of $\Gamma$ belongs to at most one edge from $M$;
 \item the graph $\Gamma'$ on the vertex set $V$ whose edge set is the union of $E\setminus M$ with the set of all edges of $M$ reversed has no directed cycles.
\end{enumerate}

The vertices that do not belong to any edge of~$M$ are called \emph{critical}. The subset of critical vertices of $X_i$ is denoted by~$X_i^M$. The key result of algebraic discrete Morse theory states that there is a way to define a new differential $d^M$ on the linear span $C^M_\bullet$ of critical vertices so that the chain complex $(C^M_\bullet,d^M)$ is quasi-isomorphic to $(C_\bullet,d)$. The only property of this differential that we shall really need is that its ``structure constants'', that is the coefficients $[v\colon v']$ in the formula
\[
 d^M(v)=\sum_{v'}[v\colon v'] v'
\]
are defined as sums over paths from $v$ to $v'$ in the graph $\Gamma'$.

\subsection{Homology classes via algebraic Morse theory}
Let $\calO=\calF_\calM/(\calG)$ be an operad generated by the collection~$\calM$ with monomial relations. Quillen homology $H^Q(\calO)$ is isomorphic to the homology of the differential induced on the space of indecomposable elements $(\calA_\calG)^{ab}$ of the operad~$\calA_\calG$. That space of generators, as a chain complex, can be decomposed into a direct sum of chain complexes $(\calA_\calG)^{ab}_T$ spanned by the elements for which the underlying tree monomial is the given tree monomial~$T$, and therefore the Quillen homology acquires a direct sum decomposition 
 \[
H^Q(\calO)=\bigoplus_{T \text{ a tree monomial}}H^Q_T(\calO).
 \]

We shall define Morse matchings of chain complexes $(\calA_\calG)^{ab}_T$ under some technical conditions which we believe to be not very restrictive; at least in all naturally arising examples that we discuss throughout the paper these conditions are fulfilled.
 
\begin{definition}
For a tree monomial $T$, we call a numbering of the set of all divisors of $T$ that are relations of $\calO$ an \emph{Anick numbering} if whenever $i<j<k$ and $S_i\cap S_j\ne\varnothing$ we have $S_i\cap S_k\subset S_j\cap S_k$. (Here and below by intersection we mean the most na\"ive combinatorial intersection of divisors inside~$T$.)
\end{definition}

Let us give two examples of Anick numberings. The first one, which we shall use as a toy model in this section, explains the term we chose: we shall see that in the case of associative algebras it corresponds to the numbering of divisors used by Anick~\cite{Anick}. 

\begin{example}\label{Anick-divisors}
Suppose that our operad $\calO$ with monomial relations is generated by unary operations. In this case, it is nothing but an associative algebra with monomial relations. If we number subwords of the given word which are relations according to the  position of the first letter, the corresponding ordering is manifestly an Anick ordering.
\end{example}

The second example is new, and has a genuine operadic meaning to it; we shall discuss an application of this result in Section~\ref{com-alg}.

We assume that our operad is generated by elements of arity~$2$. Let us use the usual terms ``left combs'' and ``right combs'' for tree monomials corresponding to the operations of the form 
 \[
\alpha_1(\alpha_2(\ldots (\alpha_k(1,i_2),i_3),\ldots, i_k),i_{k+1})\quad \text{and}\quad \alpha_1(1,\alpha_2(2,\ldots \alpha_k(k,k+1)\ldots))  
 \]
respectively: left combs are obtained from the generators by iterated compositions in the first slot, and right combs are obtained from the generators by iterated compositions in the last slot. 

\begin{proposition}\label{left-right-combs}
Suppose that $\calO=\calF_\calM/(\calG)$ is an operad with binary generators and monomial relations, $\calG=\calG_l\cup\calG_r$ where $\calG_l$ consists of left combs and $\calG_r$ consists right combs, and at least one of the sets $\calG_l$, $\calG_r$ is contained in $\calF_\calM(3)$. Then for each tree monomial $T$, the set of all divisors of $T$ that are relations admits an Anick numbering. 
\end{proposition}

\begin{proof}
Without loss of generality, $\calG_l\subset \calF_\calM(3)$. Let us define a partial ordering of the set of all divisors of $T$ that are relations of $\calO$ as follows: $S\prec S'$ if the root of $S'$ is on the path from the root of $T$ to the root of $S$, or if $S$ and $S'$ share the same root, $S$ is a left comb and $S'$ is a right comb. Let us prove that if we extend this partial ordering to a total ordering in any way, and consider the numbering of the divisors according to that total ordering in the increasing order, the numbering thus obtained is an Anick numbering. Indeed, suppose that $i<j<k$ and $S_i\cap S_j\ne\varnothing$. There are three different situations when this can happen:
\begin{itemize}
 \item[-] $S_i$ and $S_j$ are left combs, and the root of $S_j$ is on the path from the root of $T$ to the root of $S_i$,
 \item[-] $S_i$ and $S_j$ are right combs, and the root of $S_j$ is on the path from the root of $T$ to the root of $S_i$,
 \item[-] $S_i$ is a left comb, $S_j$ is a right comb, and they share a vertex which is a root vertex of $S_i$.
\end{itemize}
The only situation when the Anick numbering condition $S_i\cap S_k\subset S_j\cap S_k$ may fail is when $S_i\cap S_k\ne\varnothing$. 

In the first case, if $S_k$ is a left comb as well, the condition is manifestly fulfilled. Otherwise, if $S_k$ is a right comb, its intersection with $S_i$ consists of exactly one vertex. If that vertex is the root of $S_i$, then it is contained in $S_j$, and the condition is fulfilled. If it is not the root of $S_i$, it has to be the root of $S_k$, so the root of $S_j$ is manifestly on the path from the root of $T$ to the root of~$S_k$, so $S_k\prec S_j$, a contradiction. 

In the second case, if $S_k$ is a right comb as well, the condition is manifestly fulfilled. Otherwise, if $S_k$ is a left comb, its intersection with $S_i$ consists of exactly one vertex. If that vertex is the root of $S_i$, then it is contained in $S_j$, and the condition is fulfilled. If it is not the root of $S_i$, it has to be the root of $S_k$, so the root of $S_j$ is manifestly on the path from the root of $T$ to the root of~$S_k$, so $S_k\prec S_j$, a contradiction. 

In the third case, if $S_k$ is a right comb, then since $S_i\cap S_k\ne\varnothing$, we instantly conclude that $S_k\prec S_i$, a contradiction. If $S_k$ is a left comb, then its intersection with $S_i$ may only consist of one vertex, which is precisely the root vertex of $S_i$, that is the intersection of $S_i$ and $S_j$.
\end{proof}

Throughout this section, we always assume that we are dealing with monomial relations for which for each $T$ the set of its divisors that are relations admits an Anick numbering. Under this assumption, we shall prove the following result.

\begin{theorem}\label{representatives}
A basis of $H^Q_T(\calO)$ is in one-to-one correspondence with basis elements $v\in(\calA_\calG)^{ab}_T$ for which the following two properties hold:
\begin{itemize}
 \item[($I$)] for each $S_j$ present in $v$, we have $\partial_j(v)=0$ in $(\calA_\calG)^{ab}$; in other words, after removing $S_j$, $v$ becomes decomposable,
 \item[($II$)] for each $S_j$ not present in~$v$, there exists $i<j$ for which $\partial_i(v\wedge S_j)\ne0$ in $(\calA_\calG)^{ab}$; in other words, after marking $S_j$ in $v$ it is possible to remove the mark from the divisor $S_i$ for some $i<j$ so that the result is indecomposable.  
\end{itemize}
\end{theorem}

\begin{proof}
Let us denote by $X$ the natural basis of the chain complex $(\calA_\calG)^{ab}_T$. It gives rise to a graph $\Gamma$ reflecting the combinatorics of the differential. We shall now describe inductively a matching of the vertices of~$\Gamma$, and demonstrate that under our assumptions it is a Morse matching. Let us put $M_1$ to be the set of edges $v\to w$, where $v,w\in X$, $w=\pm\partial_1(v)$. For $k>1$ we denote by $X^{(k-1)}$ the set of critical vertices with respect to the matching $M_1\cup\ldots\cup M_{k-1}$, and let 
 \[
M_k=\{v\to w \colon v,w\in X^{(k-1)}, w=\pm\partial_k(v)\}. 
 \]

The following proposition is a ``bounded version'' of Theorem~\ref{representatives}.

\begin{proposition}
The set $X^{(k)}$ consists of the basis elements $v\in(\calA_\calG)^{ab}_T$ for which the following two properties hold:
\begin{itemize}
 \item[($I_k$)] for each $j\le k$ such that $S_j$ is present in $v$, the monomial $\partial_j(v)$ is decomposable,
 \item[($II_k$)] for each $j\le k$ such that $S_j$ is not present in~$v$, there exists $i<j$ for which $\partial_i(v\wedge S_j)\ne0$ in $(\calA_\calG)^{ab}$.  
\end{itemize}
\end{proposition}

\begin{proof}
We prove this statement by induction on~$k$. For $k=1$, it is obvious. Let us explain the step of induction. Let $v\in X^{(k)}$. By induction, Conditions ($I_{k-1}$) and ($II_{k-1}$) hold for $v$. 

Let us examine Condition ($I_k$) for a basis element $v$, and for $j=k$. Assume that it does not hold, so that $\partial_k(v)\ne0$ in $(\calA_\calG)^{ab}$. We shall now prove that in this case the elements $v$ and $\pm\partial_k v$ will have been matched when forming the matching~$M_k$. Basically, it follows from 
\begin{lemma}
Suppose that $v\in X^{(k-1)}$, and that $\partial_k(v)$ is indecomposable. Then $\pm\partial_k v$ belongs to the set of critical vertices $X^{(k-1)}$ as well.  
\end{lemma}
\begin{proof}
First, Condition $(I_{k-1})$ for $v$ implies the same condition for $\partial_k(v)$, since removing a factor $S_k$ does not ruin decomposability. Let us prove that Condition $(II_{k-1})$ holds for $\partial_k(v)$. We should check that for each $j<k$ such that $S_j$ is not present in~$\partial_k(v)$, there exists $i<j$ for which $\partial_i(\partial_k(v)\wedge S_j)$ is indecomposable. Since Condition $(II_{k-1})$ holds for $v$, for each $j<k$ not present in $v$ we can find $i<j$ such that $\partial_i(v\wedge S_j)$ is indecomposable. By $(I_{k-1})$, $\partial_i(v)$ is decomposable. Therefore, $S_i\cap S_j\ne\varnothing$, and since we work with an Anick numbering we have $S_i\cap S_k\subset S_j\cap S_k\subset S_j$. This means that $\partial_i(\partial_k(v)\wedge S_j)$ is indecomposable (the only reason for $\partial_i(v\wedge S_j)$ to become decomposable after removing $S_k$ would be that something covered by both $S_i$ and $S_k$ wasn't covered anymore, but the intersection of these divisors is covered by~$S_j$)
.
\end{proof}

Let us examine Condition ($II_k$) for a basis element $v$, and for $j=k$. Assume that it does not hold, so that $S_k$ does not occur in~$v$, and for all $i<k$ such that $S_i$ present in $v\wedge S_k$ the element $\partial_i(v\wedge S_k)$ is decomposable. We shall now show that $\pm v\wedge S_k$ and $v$ will have been matched when forming the matching $M_k$. Indeed, by our assumption Condition $(I_{k-1}$) holds for the monomial $\pm v\wedge S_k$. Condition $(II_{k-1})$ holds for this monomial trivially, since it holds for $v$, and indecomposability is preserved by the operators of wedge multiplication by $S_p$. Thus, the monomial $\pm v\wedge S_k$ belongs to $X^{(k-1)}$ by induction, and we found an edge of the matching $M_k$.

Vice versa, let us assume that Conditions ($I_k$) and ($II_k$) hold. Then Conditions ($I_{k-1}$) and ($II_{k-1}$) also hold, and so $v\in X^{(k-1)}$ by induction. If $v\notin X^{(k)}$, $v$ is used in one of the edges of the matching $M_k$. Condition ($I_k$) guarantees that no edge $v\to w$ can appear on that step, so the only option is an edge $\pm v\wedge S_k\to v$. But by condition ($II_k$), there exists $i<k$ for which $\partial_i(v\wedge S_k)$ is indecomposable, which shows that $v\wedge S_k$ would have been used at an earlier stage. 
\end{proof}

\begin{proposition}
The matching $M=\bigcup_i M_i$ is a Morse matching. 
\end{proposition}

\begin{proof}
The only condition we need to check is acyclicity, since every vertex is involved in at most one edge by the construction. Suppose that there is a directed cycle in the graph $\Gamma'$. Because it is a cycle, it has the same number of ``increasing edges'', that is reversed $M$-edges, and decreasing edges, that is edges from $E\setminus M$. For the rest of the proof, we choose an edge $e=(v\to\pm v\wedge S_k)$ of our cycle with the largest possible~$k$. 

Suppose that the edge following~$e$ is a decreasing one, that is we have a fragment $v\to\pm v\wedge S_k\to\pm\partial_l(v\wedge S_k)$
in our cycle. Clearly, $l\ne k$, since otherwise we would have the same edge belonging both to $M$ and the reversion of $M$, a contradiction. Also, it cannot be $l>k$, since otherwise we would have found an edge $u\to\pm u\wedge S_l$ elsewhere in the cycle, which would contradict the definition of~$k$. Therefore, $l<k$. But in this case, applying Condition ($I_{k-1}$) with $j=l$ to the monomial $v\wedge S_k$ we obtain a decomposable element, which is a contradiction.

Now suppose that the edge following~$e$ is an increasing one, that is we have a fragment $v\to\pm v\wedge S_k\to\pm v\wedge S_k\wedge S_l$ in our cycle. Clearly, $l\ne k$, since otherwise we have $v\wedge S_k\wedge S_l=0$. Then, according to the definition of~$k$, we have $l<k$. Let us look at the element $v\wedge S_k$. Since $S_l$ is not present in it, we apply Condition ($II_{k-1}$) with $j=l$ to this element, concluding that for some $r<l$ the monomial $\partial_r(v\wedge S_k\wedge S_l)$ is indecomposable. If we choose the smallest $r$ for which $\partial_r(v\wedge S_k\wedge S_l)$ is indecomposable, we observe that the monomials $v\wedge S_k\wedge S_l$ and $\pm\partial_r(v\wedge S_k\wedge S_l)$ were matched on the step $r$. This contradicts the fact that $v\wedge S_k\wedge S_l$ is matched with~$v\wedge S_k$. 
\end{proof}

To complete the proof of Theorem~\ref{representatives}, it is enough to show that the Morse differential $d^M$ on the critical vertices is identically zero. This would mean that the critical vertices are precisely the homology classes of the chain complex~$(\calA_\calG)^{ab}_T$. The former statement can be proved as follows. Every path between two vertices in the graph $\Gamma'$ starts either with an edge from $E\setminus M$ or with a reversed edge from $M$. No edge from $E$, in particular an edge from~$E\setminus M$ can start with a critical vertex~$v$, since Conditions $(I_k)$ altogether mean that for every $k$ the monomial $\partial_k v$ is decomposable, and consequently $d(v)=0$ in~$(\calA_\calG)^{ab}$. No reversed edge from $M$ can contain a critical vertex either, for tautological reasons. 
\end{proof}

We proceed with our examples of Anick numberings. In the case discussed in Example~\ref{Anick-divisors}, we shall, as we already mentioned, obtain Anick chains for monomial algebras~\cite{Anick,Ufn}. Let us recall their definition. Every chain is a monomial of the free algebra $\k\langle x_1,\ldots,x_n\rangle$. For $q\ge0$, $q$-chains and their tails are defined inductively as follows:
\begin{itemize}
\item[-] each generator $x_i$ is a $0$-chain; it coincides with its tail;
\item[-] each $q$-chain is a monomial $m$ equal to a product $nst$ where $t$ is the tail of $m$, and $ns$ is a $(q-1)$-chain whose tail is $s$;
\item[-] in the above decomposition, the product $st$ has exactly one divisor which is a relation of~$R$; this divisor is a right divisor of~$st$.
\end{itemize}
In other words, a $q$-chain is a monomial formed by linking one after another $q$ relations so that only neighbouring relations are linked, the first $(q-1)$ of them form a $(q-1)$-chain, and no proper left divisor is a $q$-chain. In our notation above, such a monomial $m$ corresponds to the generator $m\otimes S_1\wedge\cdots\wedge S_q$ where $S_1$, \ldots, $S_q$ are the relations we linked.

\begin{proposition}\label{indeed-Anick}
For the Anick numbering of divisors from Example~\ref{Anick-divisors}, the representatives for homology classes suggested by Theorem~\ref{representatives} are precisely Anick chains.
\end{proposition}

\begin{proof}
Indeed, condition~$(I)$ means that only neighbours are linked, and condition~$(II)$ means that no proper beginning of a $q$-chain forms a $q$-chain. 
\end{proof}

\begin{remark}
If we consider the numbering of subwords according to the position of their last letters, we obtain another Anick numbering. The fact that both of the numberings are Anick numberings can be used to obtain a conceptual proof of a result of Bardzell~\cite{Bardzell1,Bardzell2} who obzerved that ``Anick left chains'' and ``Anick right chains'' have the same set of underlying monomials, and used it to obtain a resolution of $A$ as an $A-A$-bimodule for an algebra~$A$ with monomial relations.
\end{remark}

In the setup of our second example, we shall in fact obtain a combinatorial picture modelled on Anick chains as well. Recall that we are dealing with an operad $\calO=\calF_\calM/(\calG)$ is an operad with binary generators and monomial relations all of which are left and right combs, and assume that we fix an Anick numbering of the kind described in Proposition~\ref{left-right-combs}.

\begin{definition}
To a tree monomial~$T$ made up of generators of~$\calO$, we associate a set of \emph{maximal combs}. A maximal left comb of~$T$ is a sequence of internal vertices $a_1,\ldots,a_q$ of~$T$ for which the left child of $a_q$ is a leaf, the left child of $a_l$ is $a_{l+1}$ for $1\le l\le q-1$, and the parent of $a_1$ (if any) has $a_1$ as its right vertex. Maximal right combs are defined similarly. 
\end{definition}

Clearly, for every indecomposable monomial of the operad $\calA_\calG$ each maximal left comb of the underlying tree monomial must be covered by left combs from~$\calG$, and each maximal right comb of the underlying tree monomial must be covered by right combs from~$\calG$.

\begin{definition}
A monomial in $\calA_\calG$ is said to be an \emph{Anick chain} for $\calO$ if for each of its maximal combs its covering by combs from~$\calG$ obeys the pattern governing Anick chains for associative algebras.
\end{definition}

The definitions are given in such a way that following result is proved completely analogously to Proposition~\ref{indeed-Anick}.

\begin{proposition}\label{left-right-combs-reps}
Suppose that $\calO=\calF_\calM/(\calG)$ is an operad with binary generators and monomial relations, $\calG=\calG_l\cup\calG_r$ where $\calG_l$ consists of left combs and $\calG_r$ consists right combs, and at least one of the sets $\calG_l$, $\calG_r$ is contained in $\calF_\calM(3)$. The representatives for homology classes of $\calO$ suggested by Theorem~\ref{representatives} are precisely Anick chains.
\end{proposition}

\section{Resolution for general relations}\label{resol_general}

Let $\widetilde{\calO}=\calF_\calM/(\widetilde{\calG})$ be an operad, and let $\calO=\calF_\calM/(\calG)$ be its monomial replacement, that is, $\widetilde{\calG}$ is a Gr\"obner basis of relations, and $\calG$ consists of all leading monomials of~$\widetilde{\calG}$. In Section~\ref{oper_monom}, we defined a free resolution $(\calA_\calG,d)$ for $\calO$, so that $H(\calA_\calG,d)\simeq\calO$. 

Let $\phi$ be the canonical homomorphism from $\calA_\calG$ to its homology~$\calO$ (it kills all generators of positive homological degree, and on elements of homological degree~$0$ is the canonical projection from $\calF_\calM$ to its quotient). Tree monomials that are not divisible by any of the monomial relations $\calG$ form a basis of~$\calO$, and we define a map $\pi$ as the composition of $\phi$ with the corresponding section; it sends elements of homological degree zero to their residues modulo $\calG$, represented as linear combinations of tree monomials not divisible by $\calG$ in our resolution. Since $(\calA_\calG,d)$ is a resolution of $\calO$, there exists a contracting homotopy $h$ for this resolution, so that $\left.(dh)\right|_{\ker d}= \id-\pi$ (in fact, below we shall specify a particular choice for such homotopy). Our goal is to ``deform'' this statement in the following sense. Let $\widetilde{\phi}$ be the homomorphism from $\calA_\calG$ to $\widetilde{\calO}$ that kills all generators 
of positive homological degree, and on elements of homological degree~$0$ is the canonical projection from $\calF_\calM$ to its quotient $\widetilde{\calO}=\calF_\calM/(\widetilde{\calG})$. By general results on Gr\"obner bases, tree monomials that are not divisible by any of the leading terms $\calG$ of relations $\widetilde{\calG}$ form a basis of~$\calO$ (each element $f$ of the free operad $\calF_\calM$ is represented as its residue $\overline{f}$ modulo the Gr\"obner basis $\widetilde{G}$), and we define 
a map $\widetilde{\pi}$ as the composition of $\widetilde{\phi}$ with the corresponding section; it sends elements of homological degree zero to their residues modulo $\widetilde{\calG}$, represented as linear combinations of tree monomials not divisible by $\calG$ in our resolution.

We shall prove the following result, which is essentially nothing but homological perturbation in the same way as it is used in the case of free resolutions of trivial modules over augmented associative algebras in~\cite{Anick,Kobayashi,Lambe}.

\begin{theorem}\label{deformed_diff}
There exists a ``deformed'' differential $D$ on~$\calA_\calG$ and a homotopy 
 \[
H\colon\ker D\rightarrow\calA_\calG  
 \]
such that 
 \[
H(\calA_\calG,D)\simeq\widetilde{\calO}\quad \text{and} \quad\left.(DH)\right|_{\ker D} = \id-\widetilde{\pi}.
 \]
\end{theorem}

\begin{proof}
We shall construct $D$ and $H$ simultaneously by induction. Let us introduce a partial ordering of basis elements in $\calA_\calG$ which just compares the underlying tree monomials. This partial ordering suggests the following definition: for an element $u\in \calA_\calG$, its leading term $\hat{u}$ is the part of the expansion of~$u$ as a combination of basis elements where we keep only basis elements $T\otimes S_1\wedge \cdots\wedge S_q$ with maximal possible~$T$.

If $L$ is a homogeneous linear operator on $\calA_\calG$ of some fixed (homological) degree of homogeneity (like $D$, $H$, $d$, $h$), we denote by $L_k$ the operator $L$ acting on elements of homological degree~$k$. We shall define the operators $D$ and $H$ by induction: we define the pair $(D_{k+1},H_k)$ assuming that all previous pairs are defined. At each step, we shall also be proving that 
\[
D(x)=d(\hat{x})+\text{lower terms}, \quad H(x)=h(\hat{x})+\text{lower terms},
\]
where the words ``lower terms'' refers to the partial order we defined above, meaning a linear combination of basis elements whose underlying tree monomial is smaller than the underlying tree monomial of~$\hat{x}$.
 
Basis of induction: $k=0$, so we have to define $D_1$ and $H_0$ (note that $D_0=0$ because there are no elements of negative homological degrees). In general, to define $D_l$, we should only consider the case when our element is a generator of~$\calA_\calG$, since in a dg operad the differential is defined by images of generators. For $l=1$, this means that we should consider the case where our generator corresponds to a leading monomial~$T=\lt(g)$ of some relation~$g$, and is of the form $T\otimes S$ where $S$ corresponds to the only divisor of $T$ which is a leading term, that is $T$ itself. Letting $D_1(T\otimes S)=\frac{1}{c_g}g$, where $c_g$ is the leading coefficient of~$g$, we see that  $D_1(T\otimes S)=T+\text{lower terms}$, as required. To define $H_0$, we use a yet another inductive argument, decreasing the monomials on which we want to define $H_0$. First of all, if a tree monomial~$T$ is not divisible by any of the leading terms of relations, we put~$H_0(T)=0$. Assume that $T$ is divisible by 
some leading terms of relations, and $S_1$, \ldots, $S_p$ are the corresponding divisors. Then on $\calA_\calG^T$ we can use $S_1\wedge\cdot$ as a homotopy, so $h_0(T)=T\otimes S_1$. We put 
\[
H_0(T)= h_0(T) + H_0(T- D_1h_0(T)).
\]
Here the leading term of $T-D_1h_0(T)$ is smaller than~$T$ (since we already know that the leading term of $D_1h_0(T)$ is $d_1h_0(T)=T$), so induction on the leading term applies. Note that by induction the leading term of $H_0(T)$ is $h_0(T)$. 

Suppose that $k>0$, that we know the pairs $(D_{l+1},H_l)$ for all $l<k$, and that in these degrees 
\[
D(x)=d(\hat{x})+\text{lower terms}, \quad H(x)=h(\hat{x})+\text{lower terms}.
\]
To define $D_{k+1}$, we should, as above, only consider the case of generators. In this case, we put
\[
D_{k+1}(x)=d_{k+1}(x) - H_{k-1}D_kd_{k+1}(x).
\]
The property $D_{k+1}(x)=d_{k+1}(\hat{x})+\text{lower terms}$ now easily follows by induction. To define $H_k$, we proceed in a way very similar to what we did for the induction basis. Assume that $u\in\ker D_k$, and that we know $H_k$ on all elements of $\ker D_k$ whose leading term is less than $\hat{u}$. Since $D_k(u)=d_k(\hat{u})+\text{lower terms}$, we see that $u\in\ker D_k$ implies $\hat{u}\in\ker d_k$. Then $h_k(\hat{u})$ is defined, and we put
\[
H_k(u)= h_k(\hat{u}) + H_k(u- D_{k+1}h_k(\hat{u})).
\]
Here $u-D_{k+1}h_k(\hat{u})\in\ker D_k$ and its leading term is smaller than~$\hat{u}$, so induction on the leading term applies (and it is easy to check that by induction $H_{k+1}(x)=h_{k+1}(\hat{x})+\text{lower terms}$). 

Let us check that the mappings $D$ and $H$ defined by these formulas satisfy, for each $k>0$, $D_kD_{k+1}=0$ and $\left.(D_{k+1}H_{k})\right|_{\ker D_{k}} = \id-\widetilde{\pi}$. A computation checking that is somewhat similar to the way $D$ and $H$ were constructed. Let us prove both statements simultaneously by induction. If $k=0$, the first statement is obvious. Let us prove the second one and establish that $D_1H_0(T)=(\id-\widetilde{\pi})(T)$ for each tree monomial~$T$. Slightly rephrasing that, we shall prove that for each tree monomial $T$ we have $D_1H_0(T)=T-\overline{T}$ where $\overline{T}$ is the residue of $T$ modulo~$\calG$~\cite{DK}. We shall prove this statement by induction on~$T$. If~$T$ is not divisible by any leading terms of relations, we have $H_0(T)=0=T-\overline{T}$. Let $T$ have divisors $S_1$, \ldots, $S_p$. We have $H_0(T)= h_0(T) + H_0(T- D_1h_0(T))$, so 
 \[
D_1H_0(T)= D_1h_0(T) + D_1H_0(T- D_1h_0(T)). 
 \]
By induction, we may assume that
\[
D_1H_0(T- D_1h_0(T))=T- D_1h_0(T)-\overline{(T- D_1h_0(T))}.
\]
Also, 
\[
D_1h_0(T)=D_1(T\otimes S_1)=\frac{1}{c_g}m_{T,S_1}(g)=T-r_g(T).
\] 
Here we use the usual notation for Gr\"obner bases computations \cite{DK}: $r_g(T)$ is the result of reduction of $T$ modulo $g$, and $m_{T,S_1}(g)$ denotes the result of the substitution of~$g$ into~$T$ at that place (we have $D_1(T\otimes S_1)=\frac{1}{c_g}m_{T,S_1}(g)$ since it is true when $T$ is a relation, and the differential agrees with operadic compositions). 

Combining the three previous equations, we obtain, 
\begin{multline*}
D_1H_0(T)=T-r_g(T)+ \left((T- D_1h_0(T))-\overline{(T- D_1h_0(T))}\right)=\\=T-r_g(T)+(r_g(T)-\overline{r_g(T)})=T-\overline{r_g(T)}=T-\overline{T}, 
\end{multline*}
since for a Gr\"obner basis the residue does not depend on a choice of reductions.

Assume that $k>0$, and that our statement is true for all $l<k$. We have 
 \[D_{k}D_{k+1}(x)=0\]
since
\begin{multline*}
D_{k}D_{k+1}(x)= D_k(d_{k+1}(x) - H_{k-1}D_kd_{k+1}(x))= \\=
D_kd_{k+1}(x) - D_kH_{k-1}D_kd_{k+1}(x)= D_kd_{k+1}(x) - D_kd_{k+1}(x)= 0,
\end{multline*}
because $D_kd_{k+1}k\in\ker D_{k-1}$, and so $D_kH_{k-1}(D_k(y))=D_k(y)$ by induction.
Also, for $u\in\ker D_k$ we have
 \[
D_{k+1}H_k(u)= D_{k+1}h_k(\hat{u}) + D_{k+1}H_k(u- D_{k+1}h_k(\hat{u})),
 \]
and by the induction on $\hat{u}$ we may assume that 
 \[
D_{k+1}H_k(u- D_{k+1}h_k(\hat{u}))=u- D_{k+1}h_k(\hat{u})
 \]
(on elements of positive homological degree, $\widetilde{\pi}=0$), so
 \[
D_{k+1}H_k(u)= D_{k+1}h_k(\hat{u}) + u- D_{k+1}h_k(\hat{u})=u, 
 \]
which is exactly what we need. 
\end{proof}

\section{Applications}\label{appl}

\subsection{Another proof of the PBW criterion for Koszulness}\label{PBW}

The goal of this section is to give a new proof of the Gr\"obner bases formulation~\cite{DK} version of the PBW criterion of Hoffbeck~\cite{Hoffbeck} (generalising the PBW criterion of Priddy~\cite{Priddy} for associative algebras).

\begin{theorem}\label{newPBW}
An operad with a quadratic Gr\"obner basis is Koszul. 
\end{theorem}

\begin{proof}
First of all, it is enough to prove it in the monomial case, since it gives an upper bound on the homology: for the deformed differential, the cohomology may only decrease. In the monomial quadratic case, every divisor of a tree monomial covers one internal edge, and every internal edge is covered by precisely one divisor, so all the generators of our free resolution are of homological degree one less than the number of corollas used in them, hence the homology of the bar complex is concentrated on the diagonal, and our operad is Koszul.
\end{proof}

\subsection{Operads and commutative algebras}\label{com-alg}

Recall a construction of an operad from a graded commutative algebra described in~\cite{Khor}.

Let $A$ be a connected graded associative commutative algebra. Define an operad $\calO_A$ as follows. We put  
 \[
\calO_A(I):=A_{|I|-1},
 \] 
and define, for each shuffle surjection $\phi\colon I\twoheadrightarrow\{1,\ldots,k\}$, the composition map
 \[
\circ_\phi\colon\calP(k)\otimes\calP(\phi^{-1}(1))\otimes\cdots\otimes\calP(\phi^{-1}(k))\to\calP(I)
 \]
to be the product in~$A$: 
 \[
a\circ_\phi(b_1,\ldots,b_k)=ab_1\cdots b_k.  
 \]
The arities of the elements match: since $I=\sqcup_{i=1}^k\phi^{-1}(i)$, we have $|I|-1=k-1+(|\phi^{-1}(1)|-1)+\cdots+(|\phi^{-1}(k)|-1)$. (In the symmetric case, we have also to define the actions of symmetric groups; by definition, all the components of $\calO_A$ are trivial representations of the respective symmetric groups.) 

As we remarked in \cite{DK}, a basis of the algebra~$A$ leads to a basis of the operad~$\calO_A$: product of generators of the polynomial algebra is replaced by the iterated composition of the corresponding generators of the free operad where each composition is substitution into the last slot of an operation. Assume that we know a Gr\"obner basis for the algebra~$A$ (as an associative algebra). It leads to a Gr\"obner basis for the operad~$\calO_A$ as follows: we first impose the quadratic relations defining the operad $\calO_{\k[x_1,\ldots,x_n]}$ coming from the polynomial algebra (stating that the result of a composition depends only on the operations composed, not on the order in which we compose operations), and then use the identification of relations in the polynomial algebra with elements of the corresponding operad, as above. Our next goal is to explain how to use the Anick resolution of the trivial module for~$A$ to construct a small resolution of the trivial module for~$\calO_A$. 

Let us define a collection $\calR$ which will then use to construct a resolution. We take the free operad generated by the collection $\calH$ with $s^{-1}\calH(k)\simeq H^Q_{k-2}(A)$. We take $\calR$ to be the quotient of that operad by the relations $c_1\circ_k c_2=0$, where $c_1\in\calH(k)$. In other words, in $\calR$ all compositions are allowed except for those using the last slot of an operation. We shall show that $\calR$ gives an upper bound on the Quillen homology of $\calO_A$ by proving the following theorem.

\begin{theorem}
There exists a free resolution $(\calF_\calR,d)\to \calO_A$. 
\end{theorem}

\begin{proof}
This statement is almost immediate from our previous results. Indeed, we know how to obtain a Gr\"obner basis for $\calO_A$ from a Gr\"obner basis of $A$. The leading terms of that Gr\"obner basis are all left combs 
 \[
\alpha(\beta(1,2),3) \quad \text{and} \quad \alpha(\beta(1,3),2)  
 \]
with three leaves and some right combs 
 \[
\alpha_1(1,\alpha_2(2,\ldots \alpha_k(k,k+1)\ldots)). 
 \]
By Proposition~\ref{left-right-combs-reps}, the corresponding homology classes of the associated monomial operad can be described by elements of the same shape as defined above, but we should start with the operad generated by Anick chains \cite{Anick}, not by the homology. To understand what happens in the transition from the monomial replacement to~$\calO_A$, let us look carefully into the general reconstruction scheme from the previous section. It recovers lower terms of differentials and homotopies by recalling lower terms of elements of the Gr\"obner basis. Let us do the reconstruction in two steps. At first, we shall recall all lower terms of relations except for those starting with $\alpha(\beta(\textrm{-},\textrm{-}),\textrm{-})$; the latter are still assumed to vanish. On the next step we shall recall all lower terms of those quadratic relations. Note that after the first step we model many copies of the associative algebra resolution and the differential there; so we can compute the homology 
explicitly. At the next step, a differential will be induced on this homology we computed, and we end up with a resolution of the required type. 
\end{proof}

In some cases the existence of such a resolution is enough to compute Quillen homology of~$\calO_A$; for example, it is so when the algebra~$A$ is Koszul, as we shall see now. In general, the differential of this resolution incorporates lots of information, including the higher operations (Massey (co)products) on the homology of~$A$.

Recall that if the algebra $A$ is quadratic, then the operad $\calO_A$ is quadratic as well. In~\cite{DK}, we proved that if the algebra~$A$ is PBW, then the operad~$\calO_A$ is PBW as well, and hence is Koszul. Now we shall prove the following substantial generalisation of this statement (substantially simplifying the proof of this statement given in~\cite{Khor}).

\begin{theorem}\label{O_AisKoszul}
If the algebra $A$ is Koszul, then the operad $\calO_A$ is Koszul as well. 
\end{theorem}

\begin{proof}
Koszulness of our algebra implies that the homology of the bar resolution is concentrated on the diagonal. Consequently, the operad $\calR$ constructed above is automatically concentrated on the diagonal, and so is its homology,  which completes the proof.
\end{proof}

\subsection{The operads of Rota--Baxter algebras}

The main goal of this section is to compute Quillen homology for the operad of Rota--Baxter algebras, and the operad of noncommutative Rota--Baxter algebras. Those are among the simplest examples of operads which are not covered by the Koszul duality theory, being operads with nonhomogeneous relations. Note that it is even not clear that these operads have minimal models: being operads with nontrivial unary operations, they are not covered by results of~\cite{MarklModel}, and indeed some operads with nontrivial unary operations do not admit minimal models. 

\subsubsection{The operads \texorpdfstring{$RB$} and \texorpdfstring{$ncRB$}, and their Gr\"obner bases}

\begin{definition}
A \emph{commutative Rota--Baxter algebra of weight~$\lambda$} is a vector space with an associative commutative product $a,b\mapsto a\cdot b$ and a unary operator~$P$ which satisfy the following identity:
 \[
P(a)\cdot P(b)=P(P(a)\cdot b+a\cdot P(b)+\lambda a\cdot b). 
 \]
We denote by $RB$ the operad of Rota--Baxter algebras. We view it as a shuffle operad with one binary and one unary generator.   
\end{definition}

Commutative Rota--Baxter algebras were defined in~\cite{Rota} with a motivation coming from probability theory~\cite{Baxter}. Various constructions of free commutative Rota--Baxter algebras appear in \cite{Rota,Cartier,GKe}. The latter paper also contains extensive bibliography and information on various applications of those algebras.

\begin{definition}
A \emph{noncommutative Rota--Baxter algebra of weight~$\lambda$} is a vector space with an associative product $a,b\mapsto a\cdot b$ and a unary operator~$P$ which satisfy the same identity as above:
 \[
P(a)\cdot P(b)=P(P(a)\cdot b+a\cdot P(b)+\lambda a\cdot b). 
 \]
We denote by $ncRB$ the operad of noncommutative Rota--Baxter algebras. Somehow, it is a bit simpler than the operad in the commutative case, because it can be viewed as a non-symmetric operad with one binary and one unary generator.   
\end{definition}

Noncommutative Rota--Baxter algebras has been extensively studied in the past years. We refer the reader to the paper of Ebrahimi--Fard and Guo \cite{EF-G2} for an extensive discussion of applications and occurrences of those algebras in various areas of mathematics, and a combinatorial construction of the corresponding free algebras.

Let us consider the path-lexicographic ordering of the free operad; we assume that $P>\cdot$. 

\begin{proposition}
The defining relations for operads $RB$ and $ncRB$ form a Gr\"obner basis.  
\end{proposition}

\begin{proof}
Here we present a proof for the case of $ncRB$, the proof for $RB$ is essentially the same, with the only exception that there are two S-polynomials to be reduced, as opposed to one S-polynomial in the case of $ncRB$ (which, as we pointed above, is easier because we are dealing with a non-symmetric operad).
 
For the associative suboperad of $ncRB$, the defining relations form a Gr\"obner basis, so the S-polynomials coming from the small common multiples the leading term of the associativity relation has with itself clearly can be reduced to zero. The leading term of the Rota--Baxter relation is $P(P(a_1)a_2)$. This term only has a nontrivial overlap with itself, not with the leading term of the associativity relation, and that overlap is $P(P(P(a_1)\cdot a_2)\cdot a_3)$. From this overlap, we compute the S-polynomial 
\begin{multline*}
-P(P(a_1\cdot P(a_2))\cdot a_3)-\lambda P(P(a_1\cdot a_2)\cdot a_3)+P((P(a_1)\cdot P(a_2))\cdot a_3)+\\
+P((P(a_1)\cdot a_2)\cdot P(a_3))+\lambda P((P(a_1)\cdot a_2)\cdot a_3)-P(P(a_1)\cdot a_2)\cdot P(a_3), 
\end{multline*}
and it can be reduced to zero by a lengthy sequence of reductions which we omit here (but which in fact can be read from the formula for $d\nu_3$ in Proposition~\ref{differential} below). By Diamond Lemma~\cite{DK}, our relations form a Gr\"obner basis.  
\end{proof}

\begin{remark}
In the case of the operad $ncRB$, our computation immediately provides bases for free noncommutative Rota--Baxter algebras. Indeed, since our operad is non-symmetric, the degree $n$ part of the free noncommutative Rota--Baxter algebra generated by the set $B$ is nothing but $ncRB(n)\otimes V^{\otimes n}$, where $V=\mathop{\mathrm{span}}(B)$, so we can use the above Gr\"obner basis to describe that part. More precisely, we first define the set of admissible expressions on a set $B$ recursively as follows:
\begin{itemize}
 \item elements of $B$ are admissible expressions;
 \item if $b$ is an admissible expression, then $P(b)$ is an admissible expression;
 \item if $b_1,\ldots,b_k$ are admissible expressions, and for each $i$ either $b_i$ is an element of $B$ or $b_i=P(b_i')$ with $b_i'$ an admissible expression, then their associative product $b_1\cdot b_2\cdots b_k$ is an admissible expression.
\end{itemize}
Based on this definition, we shall call some of admissible expressions the Rota--Baxter monomials, tracing the construction of an admissible expression and putting some restrictions. Namely,
\begin{itemize}
 \item elements of $B$ are Rota--Baxter monomials;
 \item if $b$ is a Rota--Baxter monomial, which, as an admissible expression, is either $b=P(b')$ or $b=b_1\cdot b_2\cdots b_k$ with $b_1\in B$, then $P(b)$ is a Rota--Baxter monomial;
 \item if $b_1,\ldots,b_k$ are Rota--Baxter monomials, and for each $i$ either $b_i\in B$ or $b_i=P(b_i')$ for some $b_i'$, then their associative product $b_1\cdot b_2\cdots b_k$ is a Rota--Baxter monomial.
\end{itemize}
Our previous discussion means that the set of all Rota--Baxter monomials forms a basis in the free noncommutative Rota--Baxter algebra generated by the set~$B$. 
It would be interesting to compare this basis with the basis from~\cite{EF-G1}.
 \end{remark}

\subsubsection{Quillen homology of the operads~\texorpdfstring{$RB$}{RB} and \texorpdfstring{$ncRB$}{ncRB}}

\begin{proposition}
For each of the operads $RB$ and $ncRB$, the resolution for its monomial version from Section~\ref{oper_monom} is minimal, that is the differential induced on the space of generators is zero. 
\end{proposition}

\begin{proof}
In the case of the operad $RB$, the overlaps obtained from the leading monomials $(a_1\cdot a_2)\cdot a_3$, $(a_1\cdot a_3)\cdot a_2$, and $P(P(a_1)\cdot a_2)$ are, in arity $n$,
 \[
((\ldots((a_1\cdot a_{i_2})\cdot a_{i_3})\cdots )\cdot a_{i_n}
 \]
and
 \[
P(P(P(\ldots P(P(P(a_1)\cdot a_{i_2})\cdot a_{i_3})\cdots )\cdot a_{i_{n-1}})\cdot a_{i_n}),
 \]
for all permutations $i_2,i_3\ldots,i_n$ of integers $2,3,\ldots,n$. It is easy to see that for each of them there exists only one indecomposable covering by relations, so the differential maps such a generator to the space of decomposable elements, and the statement follows. 

Similarly, in the case of the operad $ncRB$, the only overlaps obtained from the leading monomials $(a_1\cdot a_2)\cdot a_3$ and $P(P(a_1)\cdot a_2)$ are, in arity $n$, 
 \[
((\ldots((a_1\cdot a_2)\cdot a_3)\cdots )\cdot a_{n-1})\cdot a_n
 \]
and
 \[
P(P(P(\ldots P(P(P(a_1)\cdot a_2)\cdot a_3)\cdots )\cdot a_{n-1})\cdot a_n).
 \]
It is easy to see that for each of them there exists only one indecomposable covering by relations, so the differential maps such a generator to the space of decomposable elements, and the statement follows. 
\end{proof}

\begin{theorem}
We have 
\begin{gather*}
\dim H^Q_l(RB)(k)=
\begin{cases}
(k-1)!, l=k\ge1,\\  
(k-1)!, l=k+1\ge2.\\   
\end{cases}\\
\dim H^Q_l(ncRB)(k)=
\begin{cases}
1, l=k\ge1,\\   
1, l=k+1\ge2.  
\end{cases}
\end{gather*}
\end{theorem}

\begin{proof}
In both cases, the subspace of generators of the free resolution splits into two parts: the part obtained as overlaps of the leading terms of the associativity relations, and the part obtained as overlaps of the leading term of the Rota--Baxter relation with itself. In arity $k$, the former are all of homological degree $k-1$, while the latter --- of homological degree~$k$. This means that when we compute the homology of the differential of our resolution restricted to the space of generators, the only cancellations can happen if some of the elements resolving the associativity relation appear as differentials of some elements resolving the Rota--Baxter relation. However, since all monomials in the Rota--Baxter relation are of degree at least~$1$ in~$P$, the way the deformed differential is constructed in the proof of Theorem~\ref{deformed_diff} shows that all the terms appearing in the formulas for the respective differentials are also of degree at least~$1$ in~$P$, so no cancellations are possible.
\end{proof}

In addition to Quillen homology computation, one can ask for explicit formulas for differentials in the free resolutions. It is not difficult to write down formulas for small arities (see the example below), but in general compact formulas are yet to be found. We expect that they incorporate the Spitzer's identity and its noncommutative analogue \cite{EF-GB-P}. However, the following statement is immediate.

\begin{corollary}  
\begin{itemize}
 \item The minimal model $RB_\infty$ for the operad $RB$ is a quasi-free operad whose space of generators has a $(k-1)!$-dimensional space of generators of homological degree~$(k-2)$ in each arity~$k\ge2$, and a $(k-1)!$-dimensional space of generators of homological degree~$k-1$ in each arity~$k\ge1$.
 \item The minimal model $ncRB_\infty$ for the operad $ncRB$ is a quasi-free operad generated by operations $\mu_k$, $k\ge 2$ of arity $k$ and homological degree~$k-2$, and $\nu_l$, $l\ge 1$ of arity $l$ and homological degree~$l-1$. 
\end{itemize}
\end{corollary}

Let us conclude this section with formulas for low arities differentials in $ncRB_\infty$, to give the reader a flavour of what sort of formulas to expect.       

\begin{example}\label{differential}
We have $d\nu_1=d\mu_2=0$, and
\begin{gather*}
d\nu_2=P(\mu_2(P(\textrm{-}),\textrm{-}))+P(\mu_2(\textrm{-},P(\textrm{-})))-\mu_2(P(\textrm{-}),P(\textrm{-}))+\lambda P(\mu_2(\textrm{-},\textrm{-})),\\
d\mu_3=\mu_2(\mu_2(\textrm{-},\textrm{-}),\textrm{-})-\mu_2(\textrm{-},\mu_2(\textrm{-},\textrm{-})),
\end{gather*}
\vspace{-7mm}
\begin{multline*}
d\nu_3=\mu_3(P(\textrm{-}),P(\textrm{-}),P(\textrm{-}))-P(\mu_2(\nu_2(\textrm{-},\textrm{-}),\textrm{-})-\mu_2(\textrm{-},\nu_2(\textrm{-},\textrm{-})))
-\\-P(\mu_3(P(\textrm{-}),P(\textrm{-}),\textrm{-})+\mu_3(P(\textrm{-}),\textrm{-},P(\textrm{-}))+\mu_3(\textrm{-},P(\textrm{-}),P(\textrm{-})))+\\
+\nu_2(\mu_2(P(\textrm{-}),\textrm{-}),\textrm{-})-\nu_2(\textrm{-},\mu_2(P(\textrm{-}),\textrm{-}))+\nu_2(\mu_2(\textrm{-},P(\textrm{-})),\textrm{-})-\nu_2(\textrm{-},\mu_2(\textrm{-},P(\textrm{-})))+\\ 
+\mu_2(\nu_2(\textrm{-},\textrm{-}),P(\textrm{-}))-\mu_2(P(\textrm{-}),\nu_2(\textrm{-},\textrm{-}))+\lambda\left[\nu_2(\mu_2(\textrm{-},\textrm{-}),\textrm{-})-\nu_2(\textrm{-},\mu_2(\textrm{-},\textrm{-}))-\right.\\ \left.-
P(\mu_3(P(\textrm{-}),\textrm{-},\textrm{-})+\mu_3(\textrm{-},P(\textrm{-}),\textrm{-})+\mu_3(\textrm{-},\textrm{-},P(\textrm{-})))\right]-
\lambda^2P(\mu_3(\textrm{-},\textrm{-},\textrm{-})).
\end{multline*}
\end{example}

\subsection{The operad \texorpdfstring{$BV$}{BV} and hypercommutative algebras}\label{BV}

The main goal of this section is to explain how our results can be used to study the operad~$BV$ of Batalin--Vilkovisky algebras. The key result below (Theorem~\ref{BVGravDelta}) is also proved in~\cite{DCV}; our proofs are based on entirely different methods. 

\subsubsection{The operad \texorpdfstring{$BV$}{BV} and its Gr\"obner basis.}

Batalin--Vilkovisky algebras show up in various questions of mathematical physics. In \cite{GTV}, a cofibrant resolution for the corresponding operad was presented. However, that resolution is a little bit more that minimal. In this section, we present a minimal resolution for this operad in the shuffle category. The operad~$BV$, as defined in most sources, is an operad with quadratic--linear relations: the odd Lie bracket can be expressed in terms of the product and the unary operator. However, alternatively one can say that a $BV$-algebra is a dg commutative algebra with a unary degree~$1$ operator $\Delta$ with $\Delta^2=0$ which is a differential operator of order at most~$2$. This definition of a $BV$-algebra  is certainly not new, see, e.~g., \cite{Getzler}. With this presentation, the corresponding operad becomes an operad with homogeneous relations (of degrees~$2$ and~$3$). Our choice of degrees and signs is taken from~\cite{GTV} where it is explained how to translate between this convention and 
other popular definitions of $BV$-algebras. 

We write identities for operations evaluated on elements of degree~$0$, assuming the usual Koszul sign rule for evaluating operations on elements of arbitrary degrees. We want to emphasize that when computing Gr\"obner bases, we are dealing with operations only, and all signs arise from evaluating operations on elements. The representatives of the identities are chosen in such a way that they can be viewed as elements of the free shuffle operad; we use the language of operations, as opposed the language of tree monomials: for each~$i$, the argument~$a_i$ of an operation corresponds to the leaf~$i$ of the corresponding tree monomial. 

\begin{definition}[BV-algebras with homogeneous relations]
A \emph{Batalin-Vilkovisky algebra}, or \emph{$BV$-algebra} for short,  is a differential graded vector space $(A, d_A)$ endowed with
\begin{itemize}
\item[-] a symmetric binary product $\bullet$ of degree~$0$,
\item[-]  a unary operator $\Delta$ of degree~$+1$,
\end{itemize}
such that $(A,d_A,\Delta)$ is a mixed complex, $d_A$ is a derivation with respect to the product, and such that
\begin{itemize}
\item[-] the product $\bullet$ is associative, $(a_1\bullet a_2)\bullet a_3=a_1\bullet(a_2\bullet a_3)$ and $(a_1\bullet a_3)\bullet a_2=a_1\bullet(a_2\bullet a_3)$
\item[-] the operator $\Delta$ satisfies $\Delta^2(a_1)=0$,
\item[-] the operations satisfy the cubic identity 
\begin{multline*}
 \Delta((a_1\bullet a_2)\bullet a_3) =
\Delta(a_1\bullet a_2)\bullet a_3+\Delta(a_1\bullet a_3)\bullet a_2+a_1\bullet\Delta(a_2\bullet a_3)-\\-
(\Delta(a_1)\bullet a_2)\bullet a_3-(a_1\bullet\Delta(a_2))\bullet a_3-(a_1\bullet\Delta(a_3))\bullet a_2,
\end{multline*}
\end{itemize}
\end{definition}


Let us consider the ordering of the free operad where we first compare lexicographically the operations on the paths from the root to leaves, and then the planar permutations of leaves; we assume that $\Delta>\bullet$. 

\begin{proposition}
The above relations together with the degree~$4$ relation
\begin{multline}\label{BV4}
\Delta(a_1\bullet\Delta(a_2\bullet a_3))+\Delta(\Delta(a_1\bullet a_2)\bullet a_3)+\Delta(\Delta(a_1\bullet a_3)\bullet a_2)-\\-
\Delta((\Delta(a_1)\bullet a_2)\bullet a_3)-\Delta((a_1\bullet\Delta(a_2))\bullet a_3)-\Delta((a_1\bullet\Delta(a_3))\bullet a_2)=0
\end{multline}
form a Gr\"obner basis of relations for the operad of BV-algebras.
\end{proposition}

\begin{proof}
With respect to our ordering, the leading monomials of our original relations are $(a_1\bullet a_2)\bullet a_3$, $(a_1\bullet a_3)\bullet a_2$, $\Delta^2(a_1)$, and $\Delta(a_1\bullet(a_2\bullet a_3))$. The only small common multiple of $\Delta^2(a_1)$ and $\Delta(a_1\bullet(a_2\bullet a_3))$ gives a nontrivial S-polynomial which, is precisely the relation~\eqref{BV4}. The leading term of that relation is $\Delta(\Delta(a_1\bullet a_2) \bullet a_3)$. 

It is well known that $\dim BV(n)=2^nn!$ \cite{Getzler}, so to verify that our relations form a Gr\"obner basis, it is sufficient to show that the restrictions imposed by these leading monomials are enough. In other words, we may check that the number of arity~$n$ tree monomials that are not divisible by any of these is equal to $2^nn!$. Moreover it is sufficient to check that for $n\le 4$, since all S-polynomials of our relations will be elements of arity at most~$4$. This can be easily checked by hand, or by a computer program~\cite{DVJ}.
\end{proof}

\subsubsection{Quillen homology of the operad~\texorpdfstring{$BV$}{BV}}

Let us denote by~$\calG$ the Gr\"obner basis from the previous section.

\begin{proposition}
For the monomial replacement of $BV$, the resolution~$\calA_\calG$ from Section~\ref{oper_monom} is minimal, that is the differential induced on the space of generators is zero. 
\end{proposition}

\begin{proof}
Let us describe explicitly the space of generators, that is possible indecomposable coverings of monomials by leading terms of relations (all monomials below are chosen from the basis of the free shuffle operad, so the correct ordering of subtrees is assumed).
These are
\begin{itemize}
 \item[-] all monomials $\Delta^{k}(a_1)$, $k\ge2$ (covered by several copies of $\Delta^2(a_1)$),
 \item[-] all ``left combs''
\begin{equation}\label{Lie}
\lambda=(\ldots((a_1\bullet a_{k_2})\bullet a_{k_3})\bullet \cdots )\bullet a_{k_n}
\end{equation}
where $(k_2,\ldots,k_n)$ is a permutation of numbers $2$, \ldots, $n$, $n\ge3$ (only the leading terms $(a_1\bullet a_2)\bullet a_3$ and $(a_1\bullet a_3)\bullet a_2$ are used in the covering),
 \item[-] all the monomials
 \[
\Delta^k(\Delta(\lambda_1\bullet(\lambda_2\bullet a_j)))
 \]
where $k\ge1$, each~$\lambda_i$ is a left comb as described above (so that several copies of $\Delta^2$, the leading term of degree~$3$, and several leading terms of the associativity relations are used in the covering), 
 \item[-] all monomials
\begin{equation}
\Delta^k(\Delta(\ldots\Delta(\Delta(\lambda_1\bullet \lambda_{2})\bullet \lambda_{3})\bullet \cdots )\bullet \lambda_{n}) 
\end{equation}
where $k\ge 0$, $n\ge 3$, and $\lambda_i$ are left combs (several copies of all leading terms are used, including at least one copy of the degree~$4$ leading term).
\end{itemize}
This is a complete list of tree monomials $T$ for which $(\calA_\calG)^{ab}_T$ is nonzero in positive homological degrees. It is easy to see that for each of them there exists only one indecomposable covering by relations, that is only one generator of $\calA_\calG$ of shape $T$. Consequently, the differential maps such a generator to a combination of decomposable elements, so the differential induced on generators is identically zero.
\end{proof}

The resolution of the operad~$BV$ which one can derive by our methods from this one is quite small (in particular, smaller than the one of~\cite{GTV}) but still not minimal. However, we now have enough information to compute Quillen homology of the operad~$BV$.

\begin{theorem}\label{BVinf}
The basis of $H^Q(BV)$ is formed by monomials 
 \[
\Delta^k(a_1), \quad k\ge 1,
 \]
and all monomials of the form
 \[
\underbrace{\Delta(\ldots\Delta(\Delta(}_{n-1\textrm{ times }}\lambda_1\bullet \lambda_{2})\bullet \cdots )\bullet (\lambda_{n}\bullet a_j)), \quad n\ge 1
 \]
from the resolution of the monomial replacement of $BV$ discussed above. Here all $\lambda_i$ are left combs.
\end{theorem}

\begin{proof}
First of all, let us notice that since $\Omega(\mathbf{B}(BV))$, a free operad generated by $\mathbf{B}(BV)[-1]$, provides a resolution for $BV$, the space $H^Q(BV)[-1]$ is the space of generators of the minimal free resolution, and we shall study the resolution provided by our methods.

It is easy to check that the element $\Delta(\Delta(a_1\bullet a_2) \bullet a_3)$ that corresponds to the leading term of the only contributing S-polynomial will be killed by the differential of the element $\Delta^2(a_1\bullet(a_2\bullet a_3))$ (covered by corresponging overlapping leading terms $\Delta^2(a_1)$ and $\Delta(a_1\bullet(a_2\bullet a_3))$)  in the deformed resolution. This observation goes much further, namely
we have for $k\ge1$
\begin{multline*}
D(\Delta^k(\Delta(\ldots\Delta(\Delta(\lambda_1\bullet \lambda_{2})\bullet \lambda_{3})\bullet \cdots )\bullet (\lambda_n\bullet a_{j}))=\\ 
=\Delta^{k-1}((\Delta(\ldots\Delta(\Delta(\lambda_1\bullet \lambda_{2})\bullet \lambda_{3})\bullet \cdots )\bullet \lambda_{n})\bullet a_{j})+\text{lower terms}
\end{multline*}
in the sense of the partial ordering we discussed earlier). So, if we retain only leading terms of the differential, the resulting homology classes are represented by all the monomials of arity~$m$
 \[
\Delta(\ldots\Delta(\Delta(\lambda_1\bullet \lambda_{2})\bullet \cdots )\bullet \lambda_{n}) 
 \]
with $\lambda_n$ having at least two leaves. They all have the same homological degree~$m-2$ in the resolution (that is, formed by overlapping $m-2$ leading terms), and so there are no further cancellations. 
\end{proof}

So far we have not been able to describe a minimal resolution of the operad~$BV$ by relatively compact closed formulas, even though in principle our proof, once processed by a version of Brown's machinery~\cite{Brown,Cohen}, would clearly yield such a resolution (in the shuffle category). 

\subsubsection{The gravity operad and the Quillen homology of \texorpdfstring{$BV$}{BV}}
The gravity operad $\Grav$ and its Koszul dual $\Hycom$ were originally defined in terms of moduli spaces of curves of genus~$0$ with marked points $\calM_{0,n+1}$ \cite{Getzler1,GK}. However, we are interested in the algebraic aspects of the story, and we use the following descriptions of the gravity operad as a quadratic algebraic operad~\cite{Getzler1}. 

An algebra over the operad~$\Grav$ is a chain complex with graded antisymmetric products 
 \[[x_1,\dots,x_n]\colon A^{\otimes n}\to A\]
of degree $2-n$, which satisfy the
relations: 
\begin{multline}\label{relgrav}
\sum_{1\le i<j\le k}
\pm [[a_i,a_j],a_1,\dots,\widehat{a_i},\dots,\widehat{a_j},\dots,a_k,
b_1,\dots,b_\ell]=\\
=\begin{cases} [[a_1,\dots,a_k],b_1,\dots,b_l] , & l>0 , \\
0 , & l=0, \end{cases} 
\end{multline}
for all $k>2$, $l\ge0$, and $a_1,\dots,a_k,b_1,\dots,b_l\in A$. For example, setting $k=3$ and $l=0$, we obtain the Jacobi relation for~$[a,b]$.

Let us define an admissible ordering of the free operad whose quotient is $\Grav$ as follows. We introduce an additional weight grading, putting the weight of the corolla corresponding to the binary bracket equal to~$0$, all other weights of corollas equal to~$1$, and extending it to compositions by additivity of weight. To compare two monomials, we first compare their weights, then the root corollas, and then path sequences~\cite{DK} according to the reverse path-lexicographic order. For both of the latter steps, we need an ordering of corollas; we assume that corollas of larger arity are smaller. Then for the relation $(k,l)$ in \eqref{relgrav} (written in the shuffle notation with variables in the proper order), its leading monomial is equal to the monomial in the right hand side for $l>0$, and to the monomial $[a_1,\ldots,a_{n-2},[a_{n-1},a_n]]$ for $l=0$.

The following theorem, together with the PBW criterion, implies that the operads $\Grav$ and $\Hycom$ are Koszul, the fact first proved by Getzler~\cite{Getzler}.

\begin{theorem}\label{Grav}
For our ordering, the relations of $\Grav$ form a Gr\"obner basis of relations. 
\end{theorem}

\begin{proof}
The tree monomials that are not divisible by leading terms of relations are precisely
 \[
[\lambda_1,\lambda_2,\ldots,\lambda_{n-1},a_j],  
 \]
where all $\lambda_i$, $1\le i\le(n-1)$ are left combs as in~\eqref{Lie} (but made from brackets, not products).
\begin{lemma}\label{grav}
The graded character of the space of such elements of arity~$n$ is 
 \[
(2+t^{-1})(3+t^{-1})\cdots(n-1+t^{-1}).  
 \]
\end{lemma}
\begin{proof}
To compute the number of basis elements where the top degree corolla is of arity $k+1$ (or, equivalently, degree~$1-k$), $k\ge1$, let us notice that this number is equal to the number of basis elements
 \[
[\lambda_1,\lambda_2,\ldots,\lambda_k] 
 \]
where the arity of $\lambda_k$ is at least~$2$ (a simple bijection: join $\lambda_{n-1}$ and $a_j$ into $[\lambda_{n-1},a_j]$). The latter number is equal to
 \[
\sum_{\substack{m_1+\cdots+m_k=n, \\m_i\ge1, m_k\ge2}}\frac{(m_1-1)!(m_2-1)!\cdots(m_k-1)!m_1m_2\cdots m_k}{(m_1+m_2+\cdots+m_k)(m_2+\ldots+m_k)\cdots m_k}\binom{m_1+\cdots+m_k}{m_1,m_2,\ldots,m_k}  
 \]
where each factor $(m_i-1)!$ counts the number of left combs of arity~$m_i$, and the remaining factor is known~\cite{DVJ} to be equal to the number of shuffle permutations of the type $(m_1,\ldots,m_k)$. This can be rewritten in the form 
 \[
\sum_{m_1+\cdots+m_k=n, m_i\ge1, m_k\ge2}\frac{(m_1+\cdots+m_k-1)!}{(m_2+\cdots+m_k)(m_3+\cdots+m_k)\cdots m_k}
 \]
and if we introduce new variables $p_i=m_i+\cdots+m_k$, it takes the form
 \[
\sum_{2\le p_{k-1}<\cdots<p_1\le n-1}\frac{(n-1)!}{p_2\cdots p_k},
 \]
which clearly is the coefficient of $t^{1-k}$ in the product
\begin{multline*}
(n-1)!\left(1+\frac{1}{2t}\right)\left(1+\frac{1}{3t}\right)\cdots\left(1+\frac{1}{(n-1)t}\right)=\\=\left(2+t^{-1}\right)\left(3+t^{-1}\right)\cdots\left(n-1+t^{-1}\right). 
\end{multline*}
\end{proof}
Since the graded character of $\Grav$ is given by the same formula~\cite{Getzler1}, we indeed see that the leading terms of defining relations give an upper bound on dimensions of homogeneous components of $\Grav$ that coincides with the actual dimensions, so there is no room for further Gr\"obner basis elements. 
\end{proof}

Using the basis of $\Grav$ we just constructed, we are able to prove the following result (which gets a conceptual explanation in the next section):

\begin{theorem}\label{BVGravDelta}
On the level of collections of graded vector spaces, we have
\begin{equation}\label{BVinfty}
s H^Q(BV)\simeq\Grav^*\otimes\End_{\k s^{-1}}\oplus \delta\k[\delta], 
\end{equation}
where $\Grav^*$ is the co-operad dual to $\Grav$, $\End_{\k s^{-1}}$ is the endomorphism operad of the graded vector space $\k s^{-1}$, and $\delta\k[\delta]$ is a cofree coalgebra generated by an element $\delta$ of degree~$2$.
\end{theorem}

\begin{proof}
As above, instead of looking at the bar complex, we shall study the basis of the space of generators of the minimal resolution obtained in Theorem~\ref{BVinf}. In arity~$1$, the element $\delta^k$ (of degree $2k$) corresponds to $s \Delta^k(a_1)$ (of degree $k+(k-1)+1=2k$, the first summand coming from the fact that $\Delta$ is of degree~$1$, the second from the fact that $\Delta^k$ is an overlap of $k-1$ relations, and the last one is the degree shift given by~$s$). The case of elements of internal degree~$0$ (which in both cases are left combs) is also obvious; a left comb of arity $n$ in the space of generators of the free resolution is of homological degree~$n-2+1=n-1$, the second summand coming from the degree shift given by~$s$, and this matches the degree shift given by~$\End_{\k s^{-1}}(n)$. For elements of internal degree $k-1$, let us extract from a typical monomial 
 \[
T=\underbrace{\Delta(\ldots\Delta(\Delta(}_{k-1\textrm{ times }}\lambda_1\bullet \lambda_{2})\bullet \cdots )\bullet (\lambda_{k}\bullet a_j)), 
 \]
of this degree and of arity $n$ the left combs $\lambda_1, \lambda_2, \ldots, \lambda_{k-1}, \lambda_{k}, a_j$, and assign to~$T$ the element of $\Grav^*\otimes\End_{\k s^{-1}}$ corresponding, via the degree shift, to the element dual to the monomial~$[\lambda_1,\lambda_2,\ldots,\lambda_{k-1},\lambda_k,a_j]\in\Grav$. This establishes a degree-preserving bijection, because if arities of $\lambda_1,\ldots,\lambda_k$ are $n_1,\ldots,n_k$, the total (internal plus homological) degree of the former element is \[(k-1)+(k-2+1+(n_1-1)+\cdots+(n_k-1))+1=n+k-2\] (where we add up the $\Delta$ degree, the overlap degree, and the degree shift), and the total degree of the latter one is $(k-1)+(n-1)=n+k-2$.
\end{proof}

\subsubsection{Relationsip to Frobenius manifold construction of Barannikov and Kontsevich}

We conclude with a brief discussion on how our results match those of Barannikov and Kontsevich (\cite{BK}, see also~\cite{LS,Manin}) who proved in a rather indirect way that for a dg $BV$-algebra that satisfies the ``$\partial-\overline{\partial}$-lemma'', there exists a $\Hycom$-algebra structure on its cohomology. Their result hints that our isomorphism~\eqref{BVinfty} exists not just on the level of graded vector spaces, but rather has some deep operadic structure behind it. For precise statements and more details we refer the reader to~\cite{DCV,KSM}; the point we are trying to make here is that some known results from the symmetric category can in fact simplify some computations in the shuffle world too, predicting the shape of tree monomials in the differential of the minimal model.

From Theorem~\ref{Grav}, it follows that the operads $\Grav$ and $\Hycom$ are Koszul, so the cobar construction $\Omega(\Grav^*\otimes\End_{\k[1]})$ is a minimal model for $\Hycom$. We shall now show that the differential of $BV_\infty$ on generators coming from $\Grav^*$ deforms the differential of~$\Hycom_\infty$ in a certain sense. Let $D$ and $d$ denote the differentials of $BV_\infty$ and $\Hycom_\infty$ respectively. We can decompose $D=D_2+D_3+\cdots$ according to the homotopy co-operad structure it provides on the space of generators (note that $d=d_2$ since the operad $\Hycom$ is Koszul). Also, let $m^*$ denote the obvious coalgebra structure on $\delta\k[\delta]$. We shall call a tree monomial in $BV_\infty$ \emph{mixed}, if it contains both corollas from $\Grav^*\otimes\End_{\k[1]}$ and from $(\delta\k[\delta])$. Then $D_2=d_2+m^*$, while for $k\ge3$ the co-operation $D_k$ is zero on the generators $\delta\k[\delta]$, and maps generators from $\Grav^*$ into linear combinations of mixed tree monomials. 
Indeed, $\Hycom$-algebras are closely related to formal Frobenius manifolds, and the result of Barannikov and Kontsevich \cite{BK} essentially implies that there exists a mapping from $\Hycom$ to the homotopy quotient $BV/\Delta$. In fact, it is an isomorphism, which can be proved in several different ways, both using Gr\"obner bases and geometrically; see \cite{Markarian} for a short geometric argument proving that. This means that the following maps exist (the vertical arrows are quasiisomorphisms between the operads and their minimal models):
 \[
\xymatrix{
BV_{\infty} \ar@{->>}[d] \ar@{->>}[dr]_{\pi} & & \Hycom_\infty \ar@{->>}[d]  \\ 
BV    &BV/\Delta   & \Hycom \ar@{->}[l]_{\widetilde{\hphantom{aaaaa}}}
}
 \]
Lifting $\pi\colon BV_\infty\to BV/\Delta\simeq\Hycom$ to the minimal model $\Hycom_\infty$ of $\Hycom$, we obtain the commutative diagram
 \[
\xymatrix{
BV_{\infty} \ar@{->>}[rr]^{\psi} \ar@{->>}[d] \ar@{->>}[dr]_{\pi}& & \Hycom_\infty \ar@{->>}[d]  \\ 
BV    &BV/\Delta  & \Hycom\ar@{->}[l]_{\widetilde{\hphantom{aaaaa}}}  
}
 \]
so there exists a map of dg operads (and not just graded vector spaces, as it follows from our previous computations) between $BV_\infty$ and $\Hycom_\infty$. Commutativity of our diagram together with simple degree considerations yields what we need.

\bibliographystyle{amsalpha}
\providecommand{\bysame}{\leavevmode\hbox to3em{\hrulefill}\thinspace}

\Addresses

\end{document}